\documentclass[12pt, twoside]{article}
\usepackage{times}
\usepackage{enumerate}

\usepackage[centertags]{amsmath}
\usepackage[hyperfootnotes=false]{hyperref}
\usepackage{amsfonts}
\usepackage{amssymb}
\usepackage{amsthm}
\usepackage{newlfont}
\usepackage{amscd}
\usepackage{amsmath,amscd}

\usepackage[curve,arrow,cmactex,matrix]{xy}
\usepackage{verbatim}

\usepackage{color}

\usepackage{eucal}

\def\titlerunning#1{\gdef\titrun{#1}}
\makeatletter
\def\author#1{\gdef\autrun{\def\and{\unskip, }#1}\gdef\@author{#1}}

\def\address#1{{\def\and{\\\hspace*{18pt}}\renewcommand{\thefootnote}{}%
\footnote {#1}}
\markboth{\autrun}{\titrun}}
\makeatother

\def\subjclass#1{{\renewcommand{\thefootnote}{}%
\footnote{\emph{Mathematics Subject Classification (2010):} #1}}}
\def\keywords#1{\par\medskip
\noindent\textbf{Keywords.} #1}

\newcommand{\QQ}{\mathbb{Q}}
\newcommand{\ZZ}{\mathbb{Z}}
\newcommand{\A}{\mathbb{A}}
\newcommand{\PP}{\mathbb{P}}

\newcommand{\RR}{\mathbb{R}}

\newcommand{\FF}{\mathbb{F}}
\newcommand{\GG}{\mathbb{G}}

\newcommand{\cA}{\mathcal{A}}
\newcommand{\cB}{\mathcal{B}}
\newcommand{\m}{\mathfrak{m}}

\newcommand{\F}{\mathcal{F}}
\newcommand{\G}{\mathcal{G}}
\newcommand{\U}{\mathcal{U}}
\newcommand{\V}{\mathcal{V}}
\newcommand{\E}{\mathcal{E}}

\newcommand{\T}{\mathcal{T}}
\newcommand{\Q}{\mathcal{Q}}
\newcommand{\B}{\mathcal{B}}
\newcommand{\I}{\mathcal{I}}
\newcommand{\J}{\mathcal{J}}

\newcommand{\OO}{\mathcal{O}}
\newcommand{\Z}{\mathcal{Z}}

\newcommand{\X}{\mathcal{X}}
\newcommand{\Y}{\mathcal{Y}}

\newtheorem{thm}{Theorem}[subsection]
\newtheorem{cor}[thm]{Corollary}

\newtheorem{lem}[thm]{Lemma}
\newtheorem{prop}[thm]{Proposition}

\theoremstyle{definition}
\newtheorem{define}[thm]{Definition}
\newtheorem{const}[thm]{Construction}
\newtheorem{assume}[thm]{Assumption}

\theoremstyle{remark}
\newtheorem{rem}[thm]{Remark}

\numberwithin{equation}{section}

\makeatletter
\@addtoreset{thm}{section}
\makeatother

\numberwithin{thm}{subsection}
\numberwithin{equation}{subsection}

\DeclareMathOperator{\Ker}{Ker}

\DeclareMathOperator{\Pic}{Pic}

\DeclareMathOperator{\Br}{Br}
\DeclareMathOperator{\loc}{loc}

\DeclareMathOperator{\Tr}{Tr}
\DeclareMathOperator{\divv}{div}

\DeclareMathOperator{\ev}{ev}

\DeclareMathOperator{\df}{def}

\DeclareMathOperator{\N}{N}

\DeclareMathOperator{\disc}{disc}

\DeclareMathOperator{\val}{val}
\DeclareMathOperator{\Sel}{Sel}

\DeclareMathOperator{\spec}{spec}

\DeclareMathOperator{\cores}{cores}
\DeclareMathOperator{\Split}{split}
\DeclareMathOperator{\res}{res}

\DeclareMathOperator{\bad}{bad}
\DeclareMathOperator{\verti}{vert}

\DeclareMathOperator{\R}{R}

\def\alp{{\alpha}}
\def\bet{{\beta}}

\def\eps{{\varepsilon}}

\def\lam{{\lambda}}

\def\vphi{{\varphi}}

\def\Om{{\Omega}}

\def\Del{{\Delta}}

\def\vphi{{\varphi}}
\def\div{\divv}

\def\lrar{\longrightarrow}

\def\hrar{\hookrightarrow}

\def\x{\stackrel}

\def\ovl{\overline}
\def\what{\widehat}

\def\bksl{\;\backslash\;}

\DeclareFontEncoding{OT2}{}{} % to enable usage of cyrillic fonts
\DeclareTextFontCommand{\textcyr}{\fontencoding{OT2}\fontfamily{wncyr}\fontseries{m}\fontshape{n}\selectfont}

\newcommand{\Sha}{\textcyr{Sh}}

\begin{document}

%%%%% To ease editing, add:
\baselineskip=17pt

\titlerunning{Integral points on conic log K3 surfaces}

\title{Integral points on conic log K3 surfaces}

\author{Yonatan Harpaz}
\date{}

\maketitle

\address{Institut des Hautes Études Scientifiques, 35 Route de Chartres, 91440 Bures-sur-Yvette, France}


\begin{abstract}
Adapting a powerful method of Swinnerton-Dyer, we give explicit sufficient conditions for the existence of integral points on certain schemes which are fibered into affine conics. This includes, in particular, cases where the scheme is geometrically a smooth log K3 surface. To the knowledge of the author, this is the first family of log K3 surfaces for which such conditions are established.  
\end{abstract}

\tableofcontents

\section{Introduction}\label{s:intro}

The goal of this paper is to contribute to the Diophantine study of schemes by adapting a powerful method, pioneered by Swinnerton-Dyer, to the context of integral points. Let $S$ denote a finite set of places of $\QQ$ containing the real place. We will apply this method to give sufficient conditions for the existence of $S$-integral points on pencils of affine conics $\Y \lrar \PP^1_S$ which are determined inside the vector bundle $\OO(-n) \oplus \OO(-m)$ over $\PP^1$ by an equation of the form
\begin{equation}\label{e:main}
f(t,s)x^2 + g(t,s)y^2 = 1, 
\end{equation}
where $f(t,s),g(t,s) \in \ZZ_S[t,s]$ are separable homogeneous polynomials of degrees $2n$ and $2m$ respectively, and which split completely over $\ZZ_S$. To describe a sample result, consider the case $n=m=1$. The following statement is a special case of Theorem~\ref{t:main-3} below:

\begin{thm}\label{t:main-intro}
Let $S$ be a finite set of places of $\QQ$ containing $2,\infty$. For each $i=1,...,4$, let $c_i,d_i \in \ZZ_S$ be a pair of $S$-coprime $S$-integers such that $\Del_{i,j} := c_jd_i - c_id_j$ is non-zero for $i \neq j \in \{1,...,4\}$. Let $\Y \lrar \PP^1_{S}$ be the pencil of affine conics determined inside $\OO(-1) \oplus \OO(-1)$ by an equation of the form
\begin{equation}\label{e:in-theorem}
(c_1t + d_1s)(c_2t + d_2s)x^2 + (c_3t + d_3s)(c_4t + d_4s)y^2 = 1. 
\end{equation}
Assume that the classes $[-1],[\Del_{1,2}],[\Del_{1,3}],[\Del_{1,4}],[d_{2,3]}],[\Del_{2,4}],[\Del_{3,4}]$ form seven distinct linearly independent classes in $\mathbb{Q}^*/(\mathbb{Q}^*)^2$. Then $\Y$ has an $S$-integral point.
\end{thm}

\begin{rem}
Under the assumptions of Theorem~\ref{t:main-intro}, the surface~\ref{e:in-theorem} always has a real point and a $p$-adic point for every $p$, as well as an integral $p$-adic point for every $p \neq 2$ (see the proof of Theorem~\ref{t:main-3}). Since $2$ is assumed to be in $S$ we see that $\Y$ always has an $S$-integral adelic point.
\end{rem}

\begin{rem}
The smallest $S$-integral solution to~\ref{e:in-theorem} can have a substantially larger height than that of the coefficients $c_i,d_i$. For example, one of the smallest examples of~\ref{e:in-theorem} which satisfies the conditions of Theorem~\ref{t:main-intro} is the equation
\begin{equation}\label{e:example} 
(t + 4s)(2t + 5s)x^2 +  (3t + 2s)(5t + s)y^2 = 1
\end{equation}
whose smallest solution is $(x,y,t,s) = (35, 152, 49, -97)$. A naive heuristic for log K3 surfaces expects the height of the smallest solution to be exponential in the height of the coefficients. In particular, for coefficients even mildly larger than~\eqref{e:example}, finding a solution using a naive search is not feasible.
\end{rem}

Results concerning integral points and strong approximation on fibered varieties were first obtained by Colliot-Th\'el\`ene and Xu in~\cite{CTX11} for families of quadrics of dimension $\geq 2$. Their results were later generalized to include certain fibrations into homegenous spaces in~\cite{CTH16}, and in a different direction to remove certain non-compactness assumptions in~\cite{Xu15}. Other types of fibrations which were studied in this context include fibrations $X \lrar \A^n$ given by a norm equation of the form
\begin{equation}\label{e:norm}
\N_{L/k}(x) = f(t_1,...,t_n)
\end{equation}
where $L/k$ is a finite field extension, $x$ is a coordinate on the restriction of scalars affine space $\A^1_{L/k}$, and $f$ is a polynomial in several variables which, over the algebraic closure, factors into a product of linear polynomials. When $s=1$ and $L/k$ is a quadratic extension these are affine versions of the classical Ch\^atelet surfaces. For this case conditional results on the existence of integral points were obtained in~\cite{Gu13} under Schinzel's hypothesis. For other cases of~\eqref{e:norm} strong approximation was obtained uncondtionally in~\cite{DW13} using the method of descent. We note, however, that in all these cases the total space under consideration is log rationally connected (for example, over an algebraically closed field, the affine variety~\eqref{e:norm} contains a copy of affine space as an open subset, as soon as $f$ has a linear factor). In contrast, the surface $\Y$ of Theorem~\ref{t:main-intro} is an example of a \textbf{log K3 surface}, and, in particular, is not log rationally connected. As it is fibered into affine conics, it can be considered as the $S$-integral analogue of an elliptic K3 surface, and may hence be called a \textbf{conic log K3 surface}. As pointed out by the referee, integral points on conic log K3 surfaces were previously studied in~\cite{Be95}, which considered complements $X  = \PP^2 \setminus D$ of plane cubic curves $D \subseteq \PP^2$. While these surfaces are technically not log K3 surfaces in the somewhat restrictive sense of Definition~\ref{d:k3} below (since they are not simply connected), when $D$ is a smooth cubic such an $X$ is a log K3 surface according to other (common) definitions appearing in the literature, see Remark~\ref{r:k3}. In this case, one may fiber $X$ by an $\A^1$-family of affine conics obtained by taking the family of conics in $\PP^2$ which meet $D$ in two triple points. The assumptions of~\cite[Theorem 3.3]{Be95} guarantee in particular that this fibration has an $S$-integral section (and so $X$ has $S$-integral points), and the theorem asserts that under these assumptions $S$-integral points on $X$ are in fact Zariski dense.

\medskip

When studying $S$-integral points on an $\OO_S$-scheme $\X$ of finite type, one often begins by considering the set of $S$-integral adelic points 
$$ \X(\A_{k,S}) \x{\df}{=} \prod_{v \in S}X(k_v) \times \prod_{v \notin S} \X(\OO_v) $$ 
where $X = \X \otimes_{\OO_S} k$ is the generic fiber of $\X$. If $\X(\A_{k,S}) = \emptyset$ one may immediately deduce that $X$ has no $S$-integral points. In general, it can happen that $\X(\A_{k,S}) \neq \emptyset$ but $\X(\OO_S)$ is still empty. One way to account for this phenomenon is given by the integral version of the \textbf{Brauer-Manin obstruction}, introduced in~\cite{CTX09}. To this end one considers the set
$$ \X(\A_{k,S})^{\Br(X)} \x{\df}{=} \X(\A_{k,S}) \cap X(\A_k)^{\Br} $$
given by intersecting the set of $S$-integral adelic points with the Brauer set of $X$. When $\X(\A_{k,S})^{\Br(X)} = \emptyset$ one says that there is a Brauer-Manin obstruction to the existence of $S$-integral points. In their paper~\cite{CTX09}, Colliot-Th\'el\`ene and Xu showed that if $X$ is a homogeneous space under a simply-connected semi-simple algebraic group $G$ with connected geometric stablizers, and $G$ satisfies a certain non-compactness condition over $S$, then the Brauer-Manin obstruction is the only obstruction to the existence of $S$-integral points on $\X$. The results of~\cite{CTX09} were then extended to incorporate more general connected algebraic groups in~\cite{BD09}. Similar results hold when $X$ is a principal homogeneous space of an algebraic group of multiplicative type (see~\cite{WX12},\cite{WX13}). In~\cite{HV10} Harari and Voloch conjecture that the Brauer-Manin obstruction is the only one for $S$-integral points on open subsets of $\PP^1_{S}$, but show that this does not hold for open subsets of elliptic curves. Other counter-examples for which the Brauer-Manin obstruction is not sufficient to explain the lack of $S$-integral points are known. In some of these cases, one can still account for the lack of $S$-integral points by considering the integral Brauer set of a suitable \textbf{\'etale cover}, see~\cite[Example 5.10]{CTW12}. Other counter-examples involve an ``obstruction at infinity'', which can occur even when $X$ is geometrically very nice, for example, when $X$ is log rationally connected (see~\cite[Example 5.9]{CTW12}). For a construction of a log K3 (and in particular simply connected) surfaces $\X$ which is not obstructed at infinity, and for which $\X(\A_{k,S})^{\Br(X)} \neq \emptyset$ but $\X(\OO_S) = \emptyset$, see~\cite{Ha}.

The method of Swinnerton-Dyer which will be adapted in this paper can be considered as an extension of the \textbf{fibration method} -- a technique designed to prove the existence of rational points on a variety $X$ when $X$ is equipped with a dominant map $\pi:X \lrar B$ such that both $B$ and the generic fiber of $\pi$ are arithmetically well-behaved. For example, when $B = \PP^1_k$ and the fibers of $\pi$ above $\A^1_k$ are split (i.e., contain a geometrically irreducible open subset), one may use the fibration method in order to approximate an adelic point $(x_v) \in X(\A_{k})^{\Br}$ by an adelic point $(x_v') \in X(\A_k)$ lying over a rational point $t \in \PP^1_k(k)$ (see~\cite{Sko90}, which builds on techniques used in~\cite{CTSS87}). If more fibers are non-split, but are still geometrically split, more subtle variants of the fibration method can come into play. Such ideas were first used by Hasse in the reduction of his theorem on quadratic hypersurfaces to the $1$-dimensional case, and were extensively developed by Colliot-Th\'el\`ene, Sansuc, Serre, Skorobogatov, Swinnerton-Dyer and others (see~\cite{Se92},\cite{SD94},\cite{CT94},\cite{CTSSD98a}). The resulting methods often require the assumption of certain number theoretic conjectures, such as Schinzel's hypothesis in suitable cases or Conjecture 9.1 of~\cite{HW}. Unfortunately, these conjectures are only known to hold in very special cases (see~\cite[\S 9.2]{HW}), e.g., when $k=\QQ$ and all the non-split fibers lie over rational points.

If the fibers of $\pi$ satisfy the Hasse principle, and one manages to find an adelic point $(x_v') \in X(\A_k)$ which lies above a rational point $t \in \PP^1_k(k)$, then one may conclude that a rational point exists. Otherwise, one needs to consider the Brauer-Manin obstruction on the fibers. If the generic fiber is rationally connected (or satisfies a slightly more general condition), one may use techniques developed by Harari in~\cite{Har10} (see also~\cite{HW}) to make sure that the adelic point $(x_v') \in X_t(\A_k)$ constructed above is orthogonal to the Brauer group of the fiber $X_t$. If the generic fiber is not so well-behaved, the situation becomes considerably more complicated, and additional techniques are required. The most powerful example of such a subtle technique was pioneered by Swinnerton-Dyer, and was aimed at dealing with fibrations whose fibers are torsors under abelian varieties (under suitable conditions). We will refer to this method as the \textbf{descent-fibration} method.

The descent-fibration method first appeared in Swinnerton-Dyer's paper~\cite{SD95}, where it was applied to diagonal del-Pezzo surfaces of degree $4$. It was later generalized and established as a method to study rational points on pencils of genus $1$ curves in~\cite{CTSSD98b}. Additional applications include more general del Pezzo surfaces of degree $4$ (\cite{BSD01},\cite{CT01},\cite{Wit07}), diagonal del Pezzo surfaces of degree $3$ (\cite{SD01}), Kummer surfaces (\cite{SDS05}, \cite{HS}) and more general elliptic fibrations (\cite{CT01},\cite{Wit07}). In all these papers one is trying to establish the existence of rational points on a variety $X$. In order to apply the method one exploits a suitable geometric structure on $X$ in order to reduce the problem to the construction of rational points on a suitable fibered variety $Y \lrar \PP^1$ whose fibers are torsors under a family $A \lrar \PP^1$ of abelian varieties. The first step is to apply the fibration method above in order to find a $t \in \PP^1(k)$ such that the fiber $Y_t$ has points everywhere locally (this part typically uses the vanishing of the Brauer-Manin obstruction, and often requires Schinzel's hypothesis). The second step then consists of modifying $t$ until a suitable part of the Tate-Shafarevich group $\Sha^1(A_{t})$ vanishes, implying the existence of a $k$-rational point on $Y_{t}$. This part usually assumes, in additional to a possible Schinzel hypothesis, the finiteness of the Tate-Shafarevich group for all relevant abelian varieties, and crucially relies on the properties of the Cassels-Tate pairing.

The descent-fibration method is currently the only method powerful enough to prove (though often conditionally) the existence of rational points on families of varieties which includes non-rationally conncted varieties, such as K3 surfaces. It is hence the only source of evidence for the question of whether or not the Brauer-Manin obstruction is the only one for K3 surfaces.

The purpose of this paper is to study an adaptation of the descent-fibration method to the realm of \textbf{integral points}. To this end, we will replace torsors under abelian varieties with \textbf{torsors under algebraic tori}. In particular, we will focus our attention on torsors under \textbf{norm $1$ tori} associated to quadratic extensions. In \S\ref{s:descent} we will develop a $2$-descent formalism suitable for this context. Building on this formalism, we will adapt in \S\ref{s:main} the descent-fibration method in order to study integral points on pencils of conics of the form~\ref{e:main}. In particular, we will obtain explicit sufficient conditions for the existence of $S$-integral points on a certain natural family of conic log K3 surfaces, as described in Theorem~\ref{t:main-intro}. Apart from the application described in this paper, we expect this new integral descent-fibration method to be applicable in many other cases, opening the door for understanding integral points beyond the realm of log rationally connected schemes.

\subsection*{Acknowledgments}
The author wishes to thank the anonymous referees for their careful reading of the manuscript and their numerous useful suggestions. The author also wishes to thank Olivier Wittenberg for enlightening discussions surrounding the topic of this paper. During the writing of this paper the author was supported by the Fondation Sciences Math\'ematiques de Paris.

\section{$2$-Descent for quadratic norm 1 tori}\label{s:descent}

Let $k$ be a number field, $S_0$ a finite set of places of $k$ and $\OO_{S_0}$ the ring of $S_0$-integers of $k$. In this section we will develop a $2$-descent formalism whose goal is to yield sufficient conditions for the existence $S_0$-integral points on schemes of the form
\begin{equation}\label{e:conic-intro}
 ax^2 + by^2 = 1
\end{equation}
where $a,b$ are a mutually coprime $S_0$-integers. To this end we will study certain Tate-Shafarevich groups associated with~\ref{e:conic-intro}, and define suitable Selmer groups to compute them. The main result of this section is Corollary~\ref{c:final-step} below, which will be used in \S\ref{s:main} to find $S_0$-integral points on schemes which are fibered into affine conics of type~\ref{e:conic-intro}.

\subsection{Preliminaries}

Let $d \in \OO_{S_0}$ be a non-zero $S_0$-integer and let $K = k(\sqrt{d})$ be the associated quadratic extension. We will denote by $T_0$ the set of places of $K$ lying above $S_0$. Let us assume that $d$ satisfies the following condition:
\begin{assume}\label{a:base}
For every $v \notin S_0$ we have $\val_v(d) \leq 1$ and $\val_v(d) = 1$ if $v$ lies above $2$.
\end{assume}
The following useful lemma is by no means novel, but we could not find an explicit reference.
\begin{lem}\label{l:concrete}
If Assumption~\ref{a:base} holds then the ring $\OO_{T_0}$ is generated, as an $\OO_{S_0}$-module, by $1$ and $d$.
\end{lem}
\begin{proof}
Let $\bet \in \OO_{T_0}$ be an element and write $\bet = t + s\sqrt{d}$ with $t,s \in k$. We wish to show $\val_v(t) \geq 0$ and $\val_v(s) \geq 0$ for every $v \notin S_0$. Since $\bet \in \OO_{T_0}$ we have that $2t = \Tr_{K/k}(\bet) \in \OO_{S_0}$. Hence if $v \notin S_0$ is an odd place then $\val_v(t) \geq 0$ and since $\val_v(d) \in \{0,1\}$ it follows that $\val_v(s) \geq 0$ as well. Now assume that $v$ lies above $2$. Since we assumed $\val_v(d) = 1$ it follows that $v$ is ramified in $K$ and there exists a unique place $w$ above $v$. Furthermore, $\val_w(\sqrt{d}) = 1$. It then follows that $\val_w(t) = 2\val_v(t)$ is even and $\val_w\left(s\sqrt{d}\right) = 2\val_v(s) + 1$ is odd. This means that
$$ 0 \leq \val_w(\bet) = \val_w\left(t + s\sqrt{d}\right) = \min\left(\val_w(t),\val_w\left(s\sqrt{d}\right)\right) $$
and hence $2\val_v(t) = \val_w(t) \geq 0$ and $2\val_v(s) = \val_w(s) \geq -\val_w(\sqrt{d}) = -1$. This implies that $\val_v(t) \geq 0$ and $\val_v(s) \geq 0$, as desired.
\end{proof}

Let $\T_0$ denote the group scheme over $\OO_{S_0}$ given by the equation
$$ x^2 - dy^2 = 1 $$
We may identify the $S_0$-integral points of $\T_0$ with the set of units in $\OO_{T_0}$ whose norm is $1$ (in which case the group operation is given by multiplication in $\OO_{T_0}$). We note that technically the group scheme $\T_0$ is not an algebraic torus over $\OO_{S_0}$, since it does not split over an \'etale extension of $\OO_{S_0}$. For every divisor $a | d$ we may consider the affine $\OO_{S_0}$-scheme $\Z^{a}_0$ given by the equation
\begin{equation}\label{e:torsor}
ax^2 + by^2 = 1 
\end{equation}
where $b = -\frac{d}{a}$. Let $I_a \subseteq \OO_{T_0}$ be the $\OO_{T_0}$-ideal generated by $a$ and $\sqrt{d}$. Lemma~\ref{l:concrete} implies, in particular, that $I_a$ is generated by $a$ and $\sqrt{d}$ as an \textbf{$\OO_{S_0}$-module}. It follows that the association $(x,y) \mapsto ax + \sqrt{d}y$ identifies the set of $S_0$-integral points of $\Z^a_0$ with the set of elements in $I_a$ whose norm is $a$. We note that the norm of $I_a$ is the ideal generated by $a$, and hence we may consider the scheme $\Z^{a}_0$ above as parameterizing \textbf{generators} of $I_a$ whose norm is exactly $a$. We have a natural action of the algebraic group $\T_0$ on the scheme $\Z^{a}_0$ corresponding to multiplying a generator by a unit. Assumption~\ref{a:base} implies that $a$ and $b$ are coprime in $\OO_{S_0}$ (i.e., the ideal $(a,b) \subseteq \OO_{S_0}$ generated by $a,b$ is equal to $\OO_{S_0}$). This, in turn, implies that the action of $\T_0$ on $\Z^a_0$ exhibits $\Z^a_0$ as a \textbf{torsor} under $\T_0$, locally trivial in the \'etale topology, and hence classified by an element in the \'etale cohomology group $\alp_a \in H^1(\OO_{S_0},\T_0)$. The solubility of $\Z^a_0$ is equivalent to the condition $\alp_a = 0$. The study of $S_0$-integral points on $\Z^a_0$ hence naturally leads to the study of \'etale cohomology groups as above, just as the study of rational points on curves of genus $1$ naturally leads to Galois cohomology groups of their Jacobians. Our main goal in this paper is to construct an adaptation of Swinnerton-Dyer's method where abelian varieties and their torsors are replaced by group schemes of the form $\OO_{T_0}$ and their torsors $\Z^a_0$, respectively.

We now observe that the torsor $\Z^a_0$ is not an arbitrary torsor of $\T_0$. Since $1 \in (a,b)$ it follows that $a \in I_a^2$ and hence $I_a^2 = (a)$ by norm considerations. In particular, if $\bet = ax + \sqrt{d}y \in I_a$ has norm $a$ then $\frac{\bet^2}{a} \in \OO_d$ has norm $1$. This operation can be realized as a map of $\OO_{S_0}$-schemes
$$ q:\Z^a_0 \lrar \T_0 .$$
The action of $\T_0$ on $\Z^a_0$ is compatible with the action of $\T_0$ on itself via the multiplication-by-$2$ map $\T_0 \x{2}{\lrar} \T_0$. We will say that $q$ is a map of $\T_0$-torsors covering the map $\T_0 \x{2}{\lrar} \T_0$. It then follows that the element $\alp_a \in H^1(\OO_{S_0},\T_0)$ is a \textbf{$2$-torsion element} and we are naturally lead to study the $2$-torsion group $H^1(\OO_{S_0},\T_0)[2]$.

Finally, an obvious necessary condition for the existence of $S_0$-integral points on $\Z^a_0$ is that $\Z^a_0$ carries an $S_0$-integral \textbf{adelic point}. This condition restricts the possible elements $\alp_a$ to a suitable subgroup of $H^1(\OO_{S_0},\T_0)$, which we may call $\Sha^1(\T_0, S_0)$. We are therefore interested in studying the $2$-torsion subgroup $\Sha^1(\T_0, S_0)[2]$. It turns out that these groups are more well-behaved when the group scheme $\T_0$ is a \textbf{algebraic torus}, i.e, splits in an \'etale extension of the base ring. To this end we will temporary extend our scalars from $\OO_{S_0}$ to $\OO_S$ for a suitable finite subset $S \supset S_0$. This will be done in the next subsection. Our final goal is to prove Corollary~\ref{c:final-step}, in which we will be able to obtain information on $S_0$-integral points of the original scheme $\Z^a_0$, prior to the extension of scalars.

\subsection{The Selmer and dual Selmer groups}\label{ss:selmer}

Let $S$ be the union of $S_0$ with all the places which ramify in $K$ and all the places above $2$. We will denote by $T \subseteq \Om_{K}$ the set of places of $K$ which lie above $S$. Let $\T$ be the base change of $\T_0$ from $\OO_{S_0}$ to $\OO_{S}$. We note that $\T$ becomes isomorphic to $\GG_m$ after base changing from $\OO_{S}$ to $\OO_{T}$, and $\OO_{T}/\OO_{S}$ is an \'etale extension of rings. This means that $\T$ is an \textbf{algebraic torus} over $\OO_{S}$. We will denote by $\what{\T}$ the character group of $\T$, considered as an \'etale sheaf over $\spec(\OO_{S})$. We will use the notation $H^i(\OO_{S},\F)$ to denote \'etale cohomology of $\spec(\OO_S)$ with coefficients in the sheaf $\F$.
\begin{define}\
\begin{enumerate}[(1)]
\item
We will denote by $\Sha^1(\T, S) \subseteq H^1(\OO_{S},\T)$ the kernel of the map
$$ H^1(\OO_{S},\T) \lrar \prod_{v \in S} H^1(k_v, \T \otimes_{\OO_{S}} k_v) .$$
\item
We will denote by $\Sha^2(\what{\T}, S) \subseteq H^2(\OO_{S},\what{\T})$ the kernel of the map
$$ H^2(\OO_S,\what{\T}) \lrar \prod_{v \in S} H^2(k_v, \what{\T} \otimes_{\OO_S} k_v) .$$
\end{enumerate}
\end{define}

Since $\T$ is an algebraic torus we may apply~\cite[Theorem 4.6(a), 4.7]{Mil06} and deduce that the groups $\Sha^1(\T, S)$ and $\Sha^2(\what{\T}, S)$ are finite and that the cup product in \'etale cohomology with compact support induces a perfect pairing

\begin{equation}\label{e:pairing-sha}
\Sha^1(\T, S) \times \Sha^2(\what{\T}, S) \lrar \QQ/\ZZ.
\end{equation}

Since $2$ is invertible in $\OO_S$ the multiplication by $2$ map $\T \x{2}{\lrar} \T$ is surjective when considered as a map of \'etale sheaves on $\spec(\OO_S)$. We hence obtain a short exact sequence of \'etale sheaves
$$ 0 \lrar \ZZ/2 \lrar \T \x{2}{\lrar} \T \lrar 0 .$$
We define the \textbf{Selmer group} $\Sel(\T, S)$ to be the subgroup $\Sel(\T, S) \subseteq H^1(\OO_S,\ZZ/2)$ consisting of all elements whose image in $H^1(\OO_S,\T)$ belongs to $\Sha^1(\T, S)$. We consequently obtain a short exact sequence
$$ 0 \lrar \T_S(\OO_S)/2 \lrar \Sel(\T, S) \lrar \Sha^1(\T, S)[2] \lrar 0 $$
where $\T_S(\OO_S)/2$ denotes the cokernel of the map $\T(\OO_S) \x{2}{\lrar} \T(\OO_S)$. Similarly, we have a short exact sequence of \'etale sheaves
$$ 0 \lrar \what{\T} \x{2}{\lrar} \what{\T} \lrar \ZZ/2 \lrar 0 $$
and we define the \textbf{dual Selmer group} $\Sel(\what{\T}, S) \subseteq H^1(\OO_S,\ZZ/2)$ to be the subgroup consisting of all elements whose image in $H^2(\OO_S,\what{\T})$ belongs to $\Sha^2(\what{\T}, S)$. The dual Selmer group then sits in a short exact sequence of the form
$$ 0 \lrar H^1(\OO_S,\what{\T})/2 \lrar \Sel(\what{\T}, S) \lrar \Sha^2(\what{\T}, S)[2] \lrar 0 $$

Let us now describe the Selmer group more explicitly. The map $H^1(\OO_S,\ZZ/2) \lrar H^1(\OO_S, \T)$ can be described as follows. Since $S$ contains all the places above $2$ the Kummer sequence associated to the sheaf $\GG_m$ yields a short exact sequence
$$ 0 \lrar \OO_S^*/(\OO_S^*)^2 \lrar H^1(\OO_S,\ZZ/2) \lrar \Pic(\OO_S)[2] \lrar 0 $$
More explicitly, an element of $H^1(\OO_S,\ZZ/2)$ corresponds to a quadratic extension $K = k(\sqrt{a})$ which is unramified outside $S$. Consequently, the element $a \in k^*$, which is well-defined up to squares, must have an even valuation at every place $v \notin S$. The map $H^1(\OO_S,\ZZ/2) \lrar \Pic(\OO_S)[2]$ is then given by sending the class of $k(\sqrt{a})$ to the class of $\frac{\div(a)}{2}$, where $\div(a)$ is the divisor of $a$ when considered as a function on $\spec(\OO_S)$. Let $I_a \subseteq \OO_T$ be the ideal corresponding to the pullback of $\frac{\div(a)}{2}$ from $\OO_S$ to $\OO_T$. Then $I_a$ is an ideal of norm $(a)$ and we can form the $\OO_S$-scheme $\Z^a$ parameterizing elements of $I_a$ of norm $a$ (such a scheme admits explicit affine equations locally on $\spec(\OO_{S})$ by choosing local generators for the ideal $I_a$). The scheme $\Z^a$ is a torsor under $\T$, and the classifying class of $\Z^a$ is the image of $a$ in $H^1(\OO_S, \T)$. We note that such a scheme automatically has $\OO_v$-points for every $v \notin S$. By definition the class in $H^1(\OO_S,\ZZ/2)$ represented by $a$ belongs to $\Sel(\T, S)$ if and only if the torsor $\Z^a$ has local points over $S$, i.e., if and only if it has an $S$-integral \textbf{adelic point}.

We note that if $a$ is such that $\div(a) = 0$ on $\spec(\OO_S)$ (i.e., $a$ is an $S$-unit), then $\Z^a$ is the scheme parameterizing $T$-units in $K$ whose norm is $a$, and can be written as
\begin{equation}\label{e:torsor-2}
x^2 - dy^2 = a.
\end{equation}
Of course, this is always the case if one inverts all primes outside $S$. In particular, for every place $v \in S$, the base change of $\Z^a$ to $k_v$ becomes isomorphic to~\ref{e:torsor-2} over $k_v$, and the scheme $\Z^a$ has a $k_v$-point if and only if the Hilbert pairing $\left<a,d\right>_v \in H^2(k_v,\ZZ/2)$ vanishes. Finally, we note that if $a$ is a divisor of $d$ (so that in particular $a$ is an $S$-unit), then the scheme $\Z^a$ coincides with the base change of the our scheme of interest $\Z^a_0$ (see~\ref{e:torsor}) from $\OO_{S_0}$ to $\OO_S$.
 
On the dual side, we may consider the short exact sequence
$$ 0 \lrar \what{\T} \lrar \what{\T} \otimes \QQ \lrar \what{\T} \otimes (\QQ/\ZZ) \lrar 0 .$$
Since $\what{\T} \otimes \QQ$ is a uniquely divisible sheaf we get an identification
$$ H^2(\OO_S,\what{\T}) \cong H^1(\OO_S, \what{\T} \otimes (\QQ/\ZZ)) .$$
To compute the latter, let $\V = \R_{\OO_T/\OO_S} \A^1$ be the Weil restriction of scalars of $\A^1$ and let $\U \subseteq \V$ be the algebraic torus
$$ \U = \{x + y\sqrt{d} \in \V | x^2 - dy^2 \neq 0\} \cong R_{\OO_T/\OO_S}(\GG_m) .$$
Let $\vphi:\U \lrar \T$ denote the homomoprhism $\vphi(x+y\sqrt{d}) =\frac{x + y\sqrt{d}}{x - y \sqrt{d}}$. We then obtain a short exact sequence of algebraic tori over $\OO_S$:
$$ 1 \lrar \GG_m \lrar \U \lrar \T \lrar 1 $$
and consequently a short exact sequence of \'etale sheaves
$$ 0 \lrar \what{\T} \otimes \QQ/\ZZ \lrar \what{\U} \otimes \QQ/\ZZ \lrar \QQ/\ZZ \lrar 0 .$$
Now the character sheaf $\what{\U}$ is (non-naturally) isomorphic to the cocharacter sheaf of $\U$, which, in turn, is naturally isomorphic to $f_*\ZZ$ where $f: \spec(\OO_T) \lrar \spec(\OO_S)$ is the obvious map. By the projection formula we may then identify $\what{\U} \otimes \QQ/\ZZ$ with $f_*\QQ/\ZZ$ and since $f$ is a finite map we get that $H^1(\OO_S, \what{\U} \otimes \QQ/\ZZ) \cong H^1(\OO_T,\QQ/\ZZ)$. Finally, since the map $H^0(\OO_S,\what{\U} \otimes \QQ/\ZZ) \lrar H^0(\OO_S,\QQ/\ZZ)$ is surjective we obtain an isomorphism
$$ H^2(\OO_S,\what{\T}) \cong H^1(\OO_S,\what{\T} \otimes \QQ/\ZZ) \cong \Ker[H^1(\OO_T,\QQ/\ZZ) \x{\cores}{\lrar} H^1(\OO_S,\QQ/\ZZ)]$$
and hence an isomorphism
$$ H^2(\OO_S,\what{\T})[2] \cong \Ker[H^1(\OO_T,\ZZ/2) \x{\cores}{\lrar} H^1(\OO_S,\ZZ/2)], $$
where in both cases $\cores$ denotes the relevant corestriction map. We may hence identify $\Sha^2(\what{\T}, S)[2] \subseteq H^2(\OO_S,\what{\T})[2]$ with the subgroup of $\Ker[H^1(\OO_T,\ZZ/2) \x{\cores}{\lrar} H^1(\OO_S,\ZZ/2)]$ consisting of those extensions which furthermore split over $T$. We note that the composition $H^1(\OO_S,\ZZ/2) \lrar H^2(\OO_S,\what{\T})[2] \lrar H^1(\OO_T,\ZZ/2)$ is just the natural restriction map which sends a class $[a] \in H^1(\OO_S,\ZZ/2)$ to the class of the quadratic extension $K(\sqrt{a})$. 
We hence obtain the following explicit description of $\Sel(\what{\T}, S)$:
\begin{cor}\label{c:dual}
Let $[a] \in H^1(\OO_S,\ZZ/2)$ be a class represented by an element $a \in \OO_S$ (such that $\val_v(a)$ is even for every $v \notin S$). Then $[a] \in \Sel(\what{\T}, S)$ if and only if every place in $T$ splits in $K(\sqrt{a})$.
\end{cor}

\begin{cor}
The kernel of the map $\Sel(\what{\T}, S) \lrar \Sha^2(\what{\T}, S)[2]$ has rank $1$ and is generated by the class $[d] \in \Sel(\what{\T}, S)$.
\end{cor}

Our proposed method of $2$-descent is in essence just as way to calculate the $2$-ranks of $\Sel(\T, S)$ and $\Sel(\what{\T}, S)$. When $k=\QQ$ this method can be considered as a repackaging of Gauss' classical genus theory for computing the $2$-torsion of the class groups of quadratic fields.

The notation introduced in the next few paragraphs follows the analogous notation of~\cite{CTSSD98b} and ~\cite{CT01}. Let $S'$ be any finite set of places containing $S$ and such that $\Pic(\OO_{S'}) = 0$. For each $v \in S'$, let $V_v$ and $V^v$ denote two copies of $H^1(k_v,\ZZ/2) \cong k_v^*/(k_v^*)^2$, considered as $\FF_2$-vector spaces. We will also denote $V_{S'} = \oplus_{v \in S'} V_v$ and $V^{S'} = \oplus_{v \in S'} V^v$. By taking the sum of the Hilbert symbol pairings
\begin{equation}\label{e:pairing}
\left<,\right>_v: V_v \times V^v \lrar \ZZ/2
\end{equation}
we obtain a non-degenerate pairing
\begin{equation}\label{e:pairing-0} 
\left<,\right>_S: V_{S'} \times V^{S'} \lrar \ZZ/2
\end{equation}
Let $I_{S'}$ and $I^{S'}$ be two copies of the group $\OO_{S'}^*/(\OO_{S'}^*)^2$. As $S'$ contains all the real places and all the places above $2$ and since $\Pic(\OO_{S'}) = 0$ we have $I_{S'}, I^{S'} \cong H^1(\OO_{S'},\ZZ/2)$ and the localization maps
$$ I_{S'} \hrar V_{S'} $$
$$ I^{S'} \hrar V^{S'} $$
are injective (see~\cite[Proposition 1.1.1]{CTSSD98b}). Furthermore, Tate-Poitou's sequence for the Galois module $\ZZ/2$ implies that $I_{S'}$ is the orthogonal complement of $I^{S'}$ with respect to~\ref{e:pairing-0}, and vice versa.

For each $v \in S'$ we define a subspace $W^v \subseteq V^v$ as follows. If $v \in S$ then we let $W^v$ be the subspace generated by the class $[d]$. If $v \in S' \bksl S$ then we let $W^v$ be the image of $\OO_v^*/(\OO_v^*)^2$ inside $(k_v^*)/(k_v^*)^2$. In both cases we let $W_v \subseteq V_v$ be the orthogonal complement of $W^v$ with respect to~\ref{e:pairing}. We will denote by $W_{S'} = \oplus_{v \in S'} W_v$ and $W^{S'} = \oplus_{v \in S'} W^v$. By construction we see that $W_{S'}$ is the orthogonal complement of $W^{S'}$ with respect to~\ref{e:pairing-0} and vice versa.

Now consider the induced pairings
\begin{equation}\label{e:pairing-1} 
I_{S'} \times W^{S'} \lrar \ZZ/2
\end{equation}
and
\begin{equation}\label{e:pairing-2} 
W_{S'} \times I^{S'} \lrar \ZZ/2
\end{equation}
The $2$-descent method for calculating the ranks of $\Sel(\T,S)$ and $\Sel(\what{\T},S)$ can be neatly summarized in the following proposition (compare~\cite[Lemma 1.4.1]{CT01}):
\begin{prop}\label{p:2-descent}
The Selmer group $\Sel(\T, S)$ can be identified with each of the following groups:
\begin{enumerate}[(1)]
\item
The intersection $I_{S'} \cap W_{S'}$.
\item
The left kernel of~\ref{e:pairing-1}.
\item
The left kernel of~\ref{e:pairing-2}.
\end{enumerate}
Similarly, the dual Selmer group $\Sel(\what{\T}, S)$ can be identified with each of the following groups:
\begin{enumerate}[(1)]
\item
The intersection $I^{S'} \cap W^{S'}$.
\item
The right kernel of~\ref{e:pairing-1}.
\item
The right kernel of~\ref{e:pairing-2}.
\end{enumerate}
\end{prop}
\begin{proof}
Since $I_{S'},I^{S'}$ are orthogonal complements of each other and $W_{S'},W^{S'}$ are orthogonal complements of each other it will suffice to prove that $\Sel(\T, S) = I_{S'} \cap W_{S'}$ and $\Sel(\what{\T}, S) = I^{S'} \cap W^{S'}$. We begin by noting that since $\Pic(\OO_{S'}) = 0$ the group $H^1(\OO_S,\ZZ/2)$ can be identified with the subgroup of $I_{S'}$ consisting of those elements $a \in I_{S'}$ whose valuations are even at every $v \in S' \bksl S$. This is equivalent to saying that the local image of $a$ in $V_v$ belongs to $W_v$ for every $v \in S' \bksl S$. Similarly, we may identify $H^1(\OO_S,\ZZ/2)$ with the subgroup of $I^{S'}$ consisting of those elements whose local image in $V^v$ belongs to $W^v$ for every $v \in S' \bksl S$.

Now, given an $a \in H^1(\OO_S,\ZZ/2) \subseteq I_{S'}$ and a place $v \in S$, the condition that the local image of $a$ in $V_v$ lies in $W_v$ is by the definition the condition $\left<a,d\right>_v = 0$, which is equivalent to the existence of a $k_v$-point on $\Z^a$. It follows that $\Sel(\T, S) = I_{S'} \cap W_{S'}$.

On the dual side, given an element $a \in H^1(\OO_S,\ZZ/2) \subseteq I^{S'}$, the condition that the local image of $a$ in $V^v$ belongs to $W^v$ for $v \in S$ is equivalent to the condition that the extension $K_w(\sqrt{a})$ splits, where $w$ is the unique place of $K$ lying above $v$. It hence follows that an element $a \in H^1(\OO_S,\ZZ/2)$ lies in $W^S$ if and only if every place of $T$ splits in $K(\sqrt{a})$, and so by Corollary~\ref{c:dual} we have $\Sel(\what{\T}_S) = I^{S'} \cap W^{S'}$ as desired.
\end{proof}

\begin{rem}\label{r:ranks}
Let $S_{\Split} \subseteq S_0$ denote the places of $S_0$ which split in $K$. By Dirichlet's unit theorem it follows that $\dim_2I_{S'} = |S'|$ and by our construction it follows that $\dim_2 W^{S'} = |S'| - |S_{\Split}|$. By Proposition~\ref{p:2-descent} we may conclude that
$$ \dim_2\Sel(\T,S) - \dim_2\Sel(\what{\T},S) = \dim_2I_{S'} - \dim_2 W^{S'} = |S_{\Split}| .$$
\end{rem}

\subsection{Back to $S_0$-integral points}

We shall now turn our attention to our motivating problem, and attempt to apply the above $2$-descent formalism to obtain sufficient condition for the solubility in $\OO_{S_0}$ of our equation of interest~\ref{e:torsor} (see Corollary~\ref{c:final-step} below). We begin with the following proposition.

\begin{prop}\label{p:final-step}
Assume Condition~\ref{a:base} is satisfied and let $a | d$ be an element dividing $d$. Let $\Z^a_0$ be the $\OO_{S_0}$-scheme given by~\ref{e:torsor} and let $\Z^a = \Z^a_0 \otimes_{\OO_{S_0}} \OO_S$ be the corresponding base change. If $\Z^a$ has an $S$-integral point then $\Z^a_0$ has an $S_0$-integral point.
\end{prop}
\begin{proof}
Let $I_a$ be the $\OO_{T_0}$-ideal generated by $a,\sqrt{d}$. We need to show that $I_a$ admits a generator of norm $a$. Since $a$ is an $S$-unit the existence of an $S$-integral point on $\Z^a$ is equivalent to the existence of a $T$-unit $\bet \in \OO_T^*$ whose norm is $a$. It will hence suffice to show that such a $\bet$ must lie in $I_a$. By Condition~\ref{a:base} and the construction of $S$ we see that $\val_v(d) = 1$ for every $v \in S \bksl S_0$. It follows that every place $v \in S \bksl S_0$ is ramified in $K$. Let $S_a \subseteq S \bksl S_0$ be the subset of places where $a$ has a positive valuation and let $T_a$ be the set of places of $K$ lying above $S_a$. Then $\OO_{T_0}/I_a \cong \prod_{w \in T_a} \FF_w$ and it will suffice to show that $\val_w(\bet) > 0$ for every $w \in T_a$. But this, in turn, is true because the norm of $\bet$ is $a$.
\end{proof}

\begin{cor}\label{c:final-step}
Assume Condition~\ref{a:base} is satisfied. If $\Sel(\what{\T}, S)$ is generated by $[d]$ then for every $a | d$ the $\OO_{S_0}$-scheme $\Z^a_0$ given by~\ref{e:torsor} satisfies the $S_0$-integral Hasse principle.
\end{cor}
\begin{proof}
Assume that $\Z^a_0$ has an $S$-integral adelic point. If $\Sel(\what{\T}, S)$ is generated by $[d]$ then $\Sha^2(\what{\T}, S)[2] = 0$ and by the perfect pairing~\ref{e:pairing-sha} we may deduce that $\Sha^1(\T, S)[2] = 0$. Since $\Z^a_0$ has an $S_0$-integral adelic point the base change $\Z^a = \Z^a_0 \otimes_{\OO_{S_0}} \OO_S$ has an $S$-integral adelic point. The class of $\Z^a$ as a $\T$-torsor then lies in $\Sha^1(\T, S)[2]$, and since the latter group vanishes it follows that $\Z^a$ has an $S$-integral point. The result now follows from Proposition~\ref{p:final-step}.
\end{proof}

\subsection{The weak and strict Selmer groups}\label{ss:strict-selmer}
In this subsection we shall describe a variant of the construction above, which we refer to as the \textbf{strict Selmer group} and the \textbf{weak dual Selmer group}. The advantage of these variants over the Selmer and dual Selmer groups is that their behaviour is easier to control over families of algebraic tori as those considered in \S\ref{s:main}.

Keeping the notation of \S\ref{ss:selmer}, let $S_1 \subseteq S$ be a finite set of places, disjoint from $S_0$. Let $W^{-}_{S'} \subseteq W_{S'}$ be the subspace spanned by those tuples $(a_v)_{v \in S'}$ such that $\sum_{v \in S_1} \val_v a_v$ is even and let $W_{+}^{S'} \subseteq V^{S'}$ be the orthogonal complements of $W^{-}_{S'}$. 

\begin{define}
We define the \textbf{strict Selmer group} by 
$$ \Sel^{-}(\T,S,S_1) \x{\df}{=} I_{S'} \cap W^{-}_{S'} $$ 
and the \textbf{weak dual Selmer group} by 
$$ \Sel^{+}(\what{\T},S,S_1) \x{\df}{=} I^{S'} \cap W_{+}^{S'} .$$
\end{define}

\begin{rem}\label{r:2-descent-strict}
The left kernels of both $I_{S'} \times W_{+}^{S'} \lrar \ZZ/2$ and $W^{-}_{S'} \times I^{S'} \lrar \ZZ/2$ can be identified with the strict Selmer group. Similarly, the right kernels of these pairings can be identified with weak dual Selmer group. 
\end{rem}

\begin{rem}\label{r:ranks-2}
Let $S_{\Split} \subseteq S_0$ denote the places of $S_0$ which split in $K$. By Dirichlet's unit theorem it follows that $\dim_2I_{S'} = |S'|$ and by our construction we have $\dim_2 W_+^{S'} = |S'| - |S_{\Split}| + 1$. Using Remark~\ref{r:2-descent-strict} we may conclude that
$$ \dim_2\Sel^{-}(\T,S) - \dim_2\Sel^{+}(\what{\T},S) = \dim_2I_{S'} - \dim_2 W_+^{S'} = |S_{\Split}| - 1 .$$
\end{rem}

\begin{rem}
In light of Proposition~\ref{p:2-descent} we have natural inclusions
$$ \Sel^{-}(\T,S,S_1) \subseteq \Sel(\T,S) $$
and
$$ \Sel(\what{\T},S) \subseteq \Sel^{+}(\what{\T},S,S_1) $$
whose cokernels are at most $1$-dimensional. Furthermore, combining Remark~\ref{r:ranks} and Remark~\ref{r:ranks-2} we see that exactly one of these inclusions is an isomorphism.
\end{rem}

\section{Integral points on pencils of affine conics}\label{s:main}

In this section we will apply the $2$-descent formalism of \S\ref{ss:selmer} in order to establish sufficient conditions for the existence of $S$-integral points on certain schemes which are fibered into affine conics. Our strategy is an adaptation of Swinnerton-Dyer's descent-fibration method. To obtain unconditional results we opt not to incorporate Schinzel's hypothesis in the proof, but rely instead on a theorem of Green, Tao and Ziegler (see \S\ref{ss:gtz}). This requires, in particular, that we restrict our attention to the field $\QQ$ of rational numbers.

Let us begin by describing the type of fibered schemes we will be interested in. Let $S_0$ be a finite set of places of $\QQ$ containing the real place. Let $\I$ be a finite non-empty set of even size. For each $i \in \I$, let $c_i,d_i \in \ZZ_{S_0}$ be $S_0$-integral $S_0$-coprime elements such that 
$$ \Del_{i,j} \x{\df}{=} c_jd_i - c_id_j \neq 0 $$ 
for every $i \neq j$. Let us denote $p_i(t,s) = c_it + d_is$, and observe that $\Del_{i,j} = p_j(d_i,-c_i)$. If $\J \subseteq \I$ is a subset then we will denote $p_{\J}(t,s) = \prod_{j \in \J}p_j$.

We will denote by $\PP^1_{S_0}$ the projective line over $\spec(\ZZ_{S_0})$. Let us fix a partition $\I = \cA \cup \cB$ (i.e., $\cA \cap \cB = \emptyset$) such that $|\cA| = 2n$ and $|\cB| = 2m$ are both even. Let $\E \lrar \PP^1_{S_0}$ be the vector bundle $\OO(-n) \oplus \OO(-m) \oplus \OO(0)$. Let $a,b \in \ZZ_S$ be non-zero $S_0$-integers and let $\ovl{\Y} \subseteq \PP(\E)$ be the closed subscheme determined by the equation
\begin{equation}\label{e:conic}
a p_{\cA}(t,s)x^2 + b p_{\cB}(t,s)y^2 = z^2.
\end{equation}

The scheme $\ovl{\Y}$ is a conic bundle over $\PP^1$. The equation $z = 0$ determines a subscheme $\Z \subseteq \ovl{\Y}$ which meets each smooth fiber of $\ovl{\Y} \lrar \PP^1_{S_0}$ at two points. Our main scheme of interest in this paper is the open subscheme $\Y = \ovl{\Y} \bksl \Z$, which is fibered over $\PP^1_{S_0}$ into affine conics. More explicitly, if we let $\F \lrar \PP^1_S$ be the vector bundle $\OO(-n) \oplus \OO(-m)$ then we may recover $\Y$ as the subvariety of $\F$ given by the equation
\begin{equation}\label{e:conic-2}
ap_{\cA}(t,s)x^2 + bp_{\cB}(t,s)y^2 = 1.
\end{equation}
We will denote by $p: \Y \lrar \PP^1_{S_0}$ the associated pencil of affine conics and by $Y = \Y \otimes_{\ZZ_{S_0}} \QQ$ the base change of $\Y$ to the field of fractions $\QQ$.

When $n=1$ and $m=0$, the surface $Y$ is \textbf{log rationally connected} (see~\cite{Zh14}), and the question of integral points on $\Y$ is strongly related to the question of integral points affine 2-dimensional quadrics, as those studied in~\cite{CTX09}. The next simplest case is when $n=m=1$, in which case the surface $Y$ is a \textbf{log K3 surface}. Let us recall the definition.  
\begin{define}\label{d:k3}
A \textbf{log K3 surface} is a smooth, geometrically integral and simply connected surface $Z$ equipped with a compactification $Z \subseteq \ovl{Z}$ such that the divisor $D = \ovl{Z} \bksl Z$ is a simple normal crossing divisor and the class $[D] \in \Pic(\ovl{Z})$ is equal to the anti-canonical class of $\ovl{Z}$.
\end{define}

\begin{rem}\label{r:k3}
Definition~\ref{d:k3} is slightly more restrictive then other definitions which appear in the literature (see, e.g,~\cite{Ii79},~\cite{Zh87}). In particular, many authors do not require $Z$ to be simply-connected, but require instead the weaker property that there are no global $1$-forms on $Z$ which extend to $\ovl{Z}$ with logarithmic singuliarites along $D$. In another direction, some authors relax the condition that $[D] + K_{\ovl{Z}} = 0$ and replace it with the condition that $\dim H^0(\ovl{Z},n([D]+K_{\ovl{Z}})) = 1$ for all $n \geq 0$. In this more general context one may consider Definition~\ref{d:k3} as isolating the simplest kind of log K3 surfaces. It will be interesting to try to extend the methods for integral points to the more general setting as well.
\end{rem}

\begin{prop}\label{p:k3}
If $|\cA| = |\cB| = 2$ then the surface $Y$ is a log K3 surface.
\end{prop}
\begin{proof}
The smooth surface $Y$ is contained in the smooth proper surface $\ovl{Y}$, which is a conic bundle over $\PP^1$. When $|\cA| = |\cB| = 2$ the conic bundle $\ovl{Y}$ has exactly four singular fibers and one can easily verify that $\ovl{Y}$ is a \textbf{del Pezzo surface of degree $4$}. The divisor $D  = \ovl{Y} \bksl Y$ (given by $z=0$) is a genus $1$ curve whose class is the anti-canonical class of $\ovl{Y}$. To show that $Y$ is a log K3 surface it will hence suffice to show that $Y$ is simply connected. In light of~\cite[Lemma 3.3.6]{Ha} it will be enough to show that $\ovl{Y} \otimes_k \ovl{k}$ posses a rational curve $L$ which meets $D$ transversely in exactly one point. Let $\lam,\mu \in \ovl{\QQ}$ be such that $\disc(\lam ap_{\cA} + \mu bp_{\cB}) = 0$. Then there exists a linear homogeneous polynomial $ct + ds$ such that $\lam ap_{\cA} + \mu bp_{\cB} = (ct + ds)^2$ with $c,d \in \ovl{\QQ}$. Let $f: \PP^1 \lrar \ovl{Y} \otimes_k \ovl{k}$ be the section of the map $\ovl{Y} \lrar \PP^1$ given by $x = \sqrt{\lam}, y = \sqrt{\mu}, z = ct + ds$. Then we may take $L$ to be the image of $f$, which meets $D$ transversely at the point $f(-d:c)$.
\end{proof}

\begin{rem}
In the setting of Proposition~\ref{p:k3}, $Y$ is a log K3 surface which is the complement of an anti-canonical class divisor inside a del Pezzo surface. The work of Hassett and Tschinkel (see, in particular,~\cite[Theorem 6.18]{HT01}) implies that integral points on $Y$ are \textbf{potentially dense}, i.e., they become Zariski dense after a finite extension of the base field and a finite enlargement of the set $S_0$. 
\end{rem}

\subsection{Main results}\label{s:main}

Our goal in this section is to formulate our main theorem, giving sufficient conditions for the surface $\Y$ to have an $S_0$-integral point (see Theorem~\ref{t:main-uncond}). In order to study $S_0$-integral points on $\Y$, it will be useful to introduce a closely related $3$-dimensional scheme dominating $\Y$. Let $\X = (\A^2_{S_0} \setminus \{(0,0)\}) \times_{\PP^1_{S_0}} \Y$ be the pullback along the natural map $\A^2_{S_0} \setminus \{(0,0)\} \lrar \PP^1_{S_0}$ and let $q: \X \lrar \Y$ denote the projection on the second coordinate. We then observe that $\X$ can be identified with the subscheme of $\spec\ZZ_{S_0}[t,s,x,y]$ given by intersecting equation~\ref{e:conic-2} with the condition $(t,s) \neq (0,0)$. If $n,m > 0$ then equation~\ref{e:conic-2} already implies that $(t,s) \neq (0,0)$ and hence $\X$ is just the affine variety determined by~\ref{e:conic-2}. We note that $\X(\ZZ_{S_0})$ can be identified with the set of tuples $(t,s,x,y)$ of $S_0$-integers satisfying equation~\ref{e:conic-2} and such that $t,s$ are coprime in $\ZZ_{S_0}$. As mentioned above, if $n,m > 0$ then the last condition is automatic. We will denote by $\pi: \X \lrar \PP^1_{S_0}$ the composed map $\pi = p \circ q$ and by $X = \X \otimes_{\ZZ_{S_0}} \QQ$ the corresponding base change to $\QQ$.

\begin{rem}\label{r:smooth}
Let $S' = S_0 \cup \{2\}$. Then $\Y \otimes_{\OO_{S_0}} \OO_{S'}$ is a smooth $\OO_{S'}$-scheme and the map $p': \Y \otimes_{\OO_{S_0}} \OO_{S'} \lrar \PP^1_{S'}$ is smooth.
\end{rem}
 
\begin{rem}\label{r:X}
Since $\Pic(\ZZ_{S_0}) = 0$ and the map $q:\X \lrar \Y$ is a $\GG_m$-torsor it follows that the map
$$ \X(\ZZ_{S_0}) \lrar \Y(\ZZ_{S_0}) $$
is surjective, and hence the existence of an $S_0$-integral point on $\X$ is equivalent to the existence of an $S_0$-integral point on $\Y$. More explicitly, $S_0$-integral points of $\X$ can be pararmeterized by tuples $(t,s,x,y) \in \ZZ_{S_0}^4$ satisfying equation~\ref{e:conic-2}, and such that $t,s$ are coprime in $\ZZ_{S_0}$. Furthermore, $(t,s,x,y),(t',s',x',y') \in \X(\ZZ_{S_0})$ map to the same point of $\Y$ if and only if there exists an $S_0$-unit $u \in \ZZ_{S_0}^*$ such that $(t',s',x,y) = (ut,us,u^{-n}x,u^{-m}y)$. A similar statement holds for $k_v$-points and $\OO_v$-points of the associated local models.
\end{rem}

In order to formulate our theorem we will need some additional terminology.

\begin{define}
Given a subset $\J \subseteq \I$ we will denote by $\J^c = \I \setminus \J$ the complement of $\J$ in $\I$.
\end{define}

\begin{define}\label{d:D}
Let $d = ab$. Given a subset $\J \subseteq \I$ and an element $i \in \I$ we will denote by 
$$ D^{\J}_i = \left\{\begin{matrix} p_{\J}(d_i,-c_i) & i \notin \J \\
dp_{\J^c}(d_i,-c_i) & i \in \J \\
\end{matrix}\right. $$
and
$$ \what{D}^{\J}_i = \left\{\begin{matrix} p_{\J}(d_i,-c_i) & i \notin \J \\
-dp_{\J^c}(d_i,-c_i) & i \in \J \\
\end{matrix}\right. $$
\end{define}

\begin{rem}\label{r:complement}
By Definition~\ref{d:D} above we have $[D^{\J^c}_i] = [dD^{\J}_i] \in \QQ^*/(\QQ^*)^2$ and $[\what{D}^{\J^c}_i] = [-d\what{D}^{\J}_i] \in \QQ^*/(\QQ^*)^2$ for every $\J \subseteq \I$ and $i \in \I$. Furthermore, $D^{\emptyset}_i = \what{D}^{\emptyset}_i = 1$ and $D^{\I}_i = -\what{D}^{\I}_i = d$.
\end{rem}

\begin{define}
We will denote by $\G$ the abelian group which is a direct sum of $\QQ^*/(\QQ^*)^2$ and the $\FF_2$-vector space spanned by the formal symbols $[p_i]$ for $i \in \I$. We will denote elements in $\G$ \textbf{multiplicatively} as $[c][p_\J] = [c]\prod_{i \in \J} [p_i]$ where $\J \subseteq \I$ is some subset and $c \in \QQ^*$.
\end{define}

\begin{rem}\label{r:mult}
If $[p_{\J}] = [p_{\J'}]\cdot[p_{\J''}] \in \G$ then for every $i \in \I$ we have $[D^{\J}_i] = [D^{\J'}_i]\cdot[D^{\J''}_i] \in \QQ^*/(\QQ^*)^2$. This follows from the identities $[p_{\J}] = [p_{(\J')^c}]\cdot[p_{(\J'')^c}]$ and $[p_{\J^c}] = [p_{(\J')^c}]\cdot[p_{\J''}] = [p_{\J'}]\cdot[p_{(\J'')^c}]$.
\end{rem}

We shall now identify a few important subgroups of $\G$.
\begin{define}\label{d:G_i}
Let $i \in \I$ be an element. We will denote by $\G_i \subseteq \G$ the subgroup containing those elements $[c][p_{\J}]$ such that $|\J|$ is even and 
$$ \left[cD^{\J}_i\right] \in \left\{1,\left[aD^{\cA}_i\right]\right\} \subseteq \QQ^*/(\QQ^*)^2 .$$ 
We will denote by $\G^i \subseteq \G$ the subspace of those elements $[c][p_{\J}]$ such that $|\J|$ is even and
$$ \left[c\what{D}^{\J}_i\right] \in \left\{1,\left[aD^{\cA}_i\right]\right\} \subseteq \QQ^*/(\QQ^*)^2 .$$ 
Finally, let us denote by $\G_{D} = \cap_{i \in \I} \G_i$ and $\G^{D} = \cap_{i \in \I} \G^i$. 
\end{define}

\begin{rem}
It is not hard to verify that $\G_D$ and $\G^D$ are finite groups.
\end{rem}

We are now ready to formulate our main condition on the $\Del_{i,j}$. This condition is analogous to Condition (D) of~\cite{CTSSD98b} and~\cite{CT01}, and we hence use the same name.

\begin{assume}[Condition (D)]
The group $\G_D$ is generated by $[a][p_{\cA}]$ and $[d][p_{\I}]$ and the group $\G^D$ is generated by $[-d][p_{\I}]$.
\end{assume}

Before we state our main theorem let us recall the definition of the vertical Brauer group:
\begin{define}
Let $f: X \lrar Y$ be a map of varieties over a field $k$. The \textbf{vertical Brauer group} $\Br^{\verti}(X,f)$ is defined to be the intersection 
$$ \Br^{\verti}(X,f) \x{\df}{=} \Br(X) \cap f^*\Br(k(Y)) \subseteq \Br(k(X)) $$
\end{define}

We are now ready to state our main theorem.
\begin{thm}\label{t:main-uncond}
Let $S_0$ be a finite set of places of $\QQ$ containing $\infty$. Let $\I$ be a non-empty set of even size and for each $i \in \I$ let $c_i,d_i \in \ZZ_{S_0}$ be coprime elements such that $\Del_{i,j} := c_id_j - c_jd_i \neq 0$ whenever $i \neq j$. Let us denote $p_i(t,s) = c_it + d_is$ and $p_{\J} = \prod_{i \in \J} p_i$ for $\J \subseteq \I$ as above. Fix a pair $a,b$ of non-zero $S_0$-integers and a partition $\I = \cA \cup \cB$ into even sized subsets and consider the pencil of affine conics $\Y \lrar \PP^1_{S_0}$ given by~\ref{e:conic-2}. Assume that Condition (D) is satisfied and that there exists an $S_0$-integral adelic point $(P_v) = (t_v,s_v,x_v,y_v)$ of $\X = (\A^2_{S_0} \setminus \{(0,0)\}) \times_{\PP^1_{S_0}} \Y$ such that
\begin{enumerate}[(1)]
\item
$\val_v(dp_{\I}(t_v,s_v)) \leq 1$ for every $v \notin S_0$ and $\val_2(dp_{\I}(t,s)) = 1$ if $2 \notin S_0$.
\item
There exist two places $v^1_\infty,v^2_\infty \in S_0$ such that $-dp_{\I}(t_{v^i_{\infty}},s_{v^i_{\infty}})$ is a non-zero square in $\QQ_{v^i_\infty}$ for $i=1,2$ (note that despite the notation $v^1_\infty$ and $v^2_\infty$ are not assumed to be infinite places, just places of $S_0$).
\item
The image of $(P_v)$ in $\Y(\A_{S_0})$ is orthogonal to $\Br^{\verti}(Y,p)$.
\end{enumerate}
Then there exists an $S_0$-integral point on $\Y$.
\end{thm}

\begin{rem}
Let us say a few words on the conditions appearing in Theorem~\ref{t:main-uncond}. As mentioned above, Condition (D) is an analogue of the condition of the same name appearing in~\cite{CTSSD98b} and~\cite{CT01}. It is a completely algebraic condition, and can be shown to imply that the $2$-torsion $\Br(Y)[2]$ is contained in $\Br^{\verti}(Y,p)$, explaining to some extent the absence of any condition concerning non-vertical Brauer elements in Theorem~\ref{t:main-uncond}. Condition (1) of Theorem~\ref{t:main-uncond} is strongly related to the fact that the 2-descent formalism developed in \S\ref{ss:selmer} requires to work with norm $1$ tori associated to localizations of maximal orders. If one could extend arithmetic duality results to more general group schemes we expect this condition to become redundant. Finally, Condition (2) of Theorem~\ref{t:main-uncond} is a type of ``splitness condition'', where we want to insure the existence of certain local points on the divisor at infinity $\ovl{Y}\setminus Y$. For a brief discussion of this idea, see~\cite[Definition 1.0.6, Definition 1.0.8]{Ha}.
\end{rem}

The conditions of Theorem~\ref{t:main-uncond} are finitely verifiable, but are somewhat inexplicit. The following lemma identifies certain stronger conditions which are easier to check ``by hand''.
\begin{lem}\label{l:explicit}
Consider the following conditions. 
\begin{enumerate}[(1)]
\item
Let $U$ be the quotient of the $\FF_2$-vector space $\QQ^*/(\QQ^*)^2$ by the subspace spanned by $\{[a],[b]\}$. Then 
the subspace $V \subseteq U$ spanned by the images of $\{[\Del_{i,j}]\}_{i \neq j}$ attains its maximal dimension $1 + \left(\begin{matrix} |\I| \\ 2\end{matrix}\right)$. In other words, all the linear relations are spanned by $[\Del_{i,j}] = [-1][\Del_{j,i}]$.
\item
For every $\J \subseteq \I$ such that $\J \neq \emptyset,\I$ and $|\J|$ is even, the subspace $V_{\J} \subseteq \QQ^*/(\QQ^*)^2$ spanned by $\{[aD^{\J}_i]\}_{i \in \I}$ attains its maximal dimension $|\I| - 1$. In other words, the only non-trivial linear relation is $\prod_{i\in \I} [aD^{\J}_i] = 1$.
\item
Condition (D) holds and $\Br^{\verti}(Y,p)/\Br(k) = 0$.
\end{enumerate}
Then we have $(1) \Rightarrow (2) \Rightarrow (3)$.
\end{lem}
\begin{proof}
$(1) \Rightarrow (2)$.
Let $\J \subseteq \I$ be such that $\J \neq \emptyset,\I$ and $|\J|$ is even. Let $\eps_i \in \ZZ/2$ for $i \in \I$ be such that $\prod_{i\in \I} [aD^{\J}_i]^{\eps_i} = 1$. Then
$$ \prod_{i\in \I} [aD^{\J}_i]^{\eps_i} = \left[\prod_{i\in \J^c} [aD^{\J}_i]^{\eps_i}\right]\left[\prod_{j \in \J}([bD^{\J^c}_j])^{\eps_j}\right] = \left[\prod_{i \in \J^c} [a]^{\eps_i}\right]\left[\prod_{j \in \J} [b]^{\eps_j}\right]\left[\prod_{i \in \J^c, j \in \J} [\Del_{i,j}]^{\eps_i}[\Del_{j,i}]^{\eps_j}\right] = $$
$$ \left[\prod_{i \in \J^c} [a]^{\eps_i}\right]\left[\prod_{j \in \J} [b]^{\eps_j}\right]\left[\prod_{i \in \J^c, j \in \J} [-1]^{\eps_j}[\Del_{i,j}]^{\eps_i+\eps_j}\right] =\left[\prod_{i \in \J^c} [a]^{\eps_i}\right]\left[\prod_{j \in \J} [b]^{\eps_j}\right]\left[\prod_{i \in \J^c, j \in \J} [\Del_{i,j}]^{\eps_i+\eps_j}\right] $$
and so Condition (1) implies that $\eps_i + \eps_j = 0$ for every $i \in \J^c, j \in \J$. Since $\J \neq \emptyset,\I$ we may conclude that all the $\eps_i$'s are equal, and so Condition (2) holds.

$(2) \Rightarrow (3)$. Let us first show that condition (D) is satisfied. Let $x = [c][p_{\J}] \in \G$ be such that $|\J|$ is even. By (2) we know that if $\J \neq \emptyset,\I$, then the values $\{[aD^{\J}_i]\}_{i \in \I}$ are linearly independent, and so, in particular, the values $\{cD^{\J}_i\}_{i \in \I}$ are pairwise distinct. On the other hand, if $\J \in \{\emptyset,\I\}$ then $[cD^{\J}_i]$ is the same for all $i \in \I$. Now let $x = [c][p_{\J}]$ be such that that $|\J|$ is even and $[cD^{\J}_i] \in \{1,[aD^{\cA}_i]\}$ for every $i \in \I$. If $[cD^{\J}_i] = 1$ for two distinct $i$'s then by the above we have $\J \in \{\emptyset,\I\}$, in which case either $x = [1]$ or $x = [d][p_{\I}]$. On the other hand, if $[cD^{\J}_i] = [aD^{\cA}_i]$ for two distinct $i$'s then by Remark~\ref{r:mult} we may apply the above argument to $x[a][p_{\cA}]$ and deduce that $x = [a][p_{\cA}]$ or $x = [a][p_{\cA}]\cdot[d][p_{\I}]$. Hence the only case left to consider is the case where $[cD^{\J}_i] = 1$ for exactly one $i$ and $[cD^{\J}_i] = [aD^{\cA}_i]$ for exactly one $i$. In this case $\I$ must be of size $2$, and since $[cD^{\J}_i]$ is not constant we know by the above that $\J$ cannot be $\emptyset$ or $\I$. Fortunately, this case is excluded thanks to our assumption that $|\J|$ is even. We may hence conclude that Condition (D) is satisfied. It is left to show that $\Br^{\verti}(Y,p)$ contains only constant classes. But this follows from the explicit computation of the vertical Brauer group (see Proposition~\ref{p:vertical} below) since $\{[aD^{\cA}_i]\}_{i \in \I}$ spans a space of dimension $|\I| - 1$.
\end{proof}

Lemma~\ref{l:explicit} can be be used to offer the following explicit variant of Theorem~\ref{t:main-uncond}.
\begin{thm}\label{t:main-2}
Let $S_0$ be a finite set of places of $\QQ$ containing $\infty$ and $2$. Let $\I$ be a non-empty set of even size and for each $i \in \I$ let $c_i,d_i \in \ZZ_{S_0}$ be coprime elements such that $\Del_{i,j} = c_id_j - c_jd_i \neq 0$ whenever $i \neq j$. Let $a,b$ be non-zero $S_0$-integers such that Condition (1) of Lemma~\ref{l:explicit} holds. Let $\I = \cA \cup \cB$ be a partition into even sized subsets and let $\Y \lrar \PP^1_{S_0}$ be the pencil of affine conics given by~\ref{e:conic-2}. Assume that for every place $v \notin S_0$ such that either $|v| < |\I|$ or $v | d$, there exists a point $(t_v,s_v,x_v,y_v) \in \X(\OO_v)$ such that $\val_v(dp_{\I}(t_v,s_v)) \leq 1$. Then $\Y$ has an $S_0$-integral point.
\end{thm}

\begin{proof}[Proof of Theorem~\ref{t:main-2} assuming Theorem~\ref{t:main-uncond}]
By Lemma~\ref{l:explicit} we know that Condition (D) holds. Let $\X = (\A^2_{S_0} \setminus \{(0,0)\}) \times_{\PP^1_{S_0}} \Y$. We shall now construct an adelic point $(P_v) = (t_v,s_v,x_v,y_v) \in \X(\A_{S_0})$ satisfying Conditions (1) and (2) of Theorem~\ref{t:main-uncond}. By Lemma~\ref{l:explicit} we have $\Br^{\verti}(Y,p)/\Br(k) = 0$ and hence condition (3) holds automatically.

Since $d \neq 0$ and the polynomial $p_{\I}(t,s)$ is separable the equation $z^2 = -dp_{\I}(t,s)$ defines a smooth hyperelliptic $C$ curve over $\QQ$. Since $p_{\I}(t,s)$ has roots in $\QQ$ this hyperelliptic curves has $\QQ$-points, and hence $\QQ_v$-points for every place $v$. It follows that $C(\QQ_v)$ is infinite for every $v$, and hence we can choose, for every $v \in S_0$, elements $t_v,s_v \in \QQ_v$ such that $-dp_{\I}(t_v,s_v)$ is a non-zero square. We then have that $\left<ap_{\cA}(t_v,s_v),bp_{\cB}(t_v,s_v)\right> = 0$ and hence the fiber $X_{(t_v,s_v)}$ has a $\QQ_v$-point. Choosing the $S_0$-components of our adelic point $(P_v)$ in this way, and using the fact that $S_0$ contains at least two places ($\infty$ and $2$), we may insure that $(P_v)$ will satisfy condition (2) of Theorem~\ref{t:main-uncond}.

Let us now choose the components of $(P_v)$ for $v \notin S_0$. If $|v| < |\I|$ or $v | d$ then our assumptions guarantee the existence of a suitable local point. Now let $v \notin S_0$ be such that $|v| \geq |\I|$ and $v$ does not divide $d$. Then $\PP^1(\FF_v)$ has more than $|\I|$ points and so there must exist a pair $t_v,s_v \in \OO_v$ such that $\val_v(ap_{\cA}(t_v,s_v))= \val_v(bp_{\cB}(t_v,s_v)) = 0$. Since $2 \in S_0$ we have that $v$ is odd and hence $\left<ap_{\cA}(t_v,s_v),bp_{\cB}(t_v,s_v)\right> = 0$ and the fiber $X_{(t_v,s_v)}$ has an $\OO_v$-point. Furthermore, by construction the $t_v,s_v$ coordinates of this point satisfy $\val_v(dp_{\I}(t_v,s_v)) = 0$.
\end{proof}

Finally, when $|\I| \leq 4$, an additional simplification is possible.
\begin{thm}\label{t:main-3}
Let $S_0$ be a finite set of places of $\QQ$ containing $\infty$ and $2$. Let $\I$ be a non-empty set of even size $\leq 4$ and for each $i \in \I$ let $c_i,d_i \in \ZZ_{S_0}$ be coprime elements such that $\Del_{i,j} = c_id_j - c_jd_i \neq 0$ whenever $i \neq j$. Let $a,b$ be \textbf{square-free} $S_0$-integers such that Condition (1) of Lemma~\ref{l:explicit} holds. Let $\I = \cA \cup \cB$ be a partition into even sized subsets and let $\Y \lrar \PP^1_{S_0}$ be the pencil of affine conics given by~\ref{e:conic-2}. Assume that $d$ is not divisible by any prime in $\{3,5\} \cap (\Om_k \setminus S_0)$. If $\Y$ contains an $\OO_v$-point for every $v \notin S_0$ which divides $d$ then $\Y$ contains an $S_0$-integral point.  
\end{thm}

\begin{proof}[Proof of Theorem~\ref{t:main-3} assuming Theorem~\ref{t:main-2}]
Assume that $\Y$ has an $S_0$-integral adelic point. We need to show that the assumptions of Theorem~\ref{t:main-2} are satisfied. First let $v \notin S_0$ be a place such that $v | d$. Since $d = ab$ is square-free we have either $v | a$ or $v | b$. Without loss of generality we may assume that $v | a$ and $v \not{|} b$. Note that by our assumption $|v| > 5$. Now consider the (possibly singular) curve $C$ over $\FF_v$ given by the hyperellliptic equation $bp_{\cB}(t,s) = u^2$. Since $a$ reduces to $0$ mod $v$ we see that the $\FF_v$-points of $Y$ are in bijection with the subset of $\FF_v$-point of $C$ for which $u \neq 0$. We note that our assumptions imply that $Y$ contains $\FF_v$-points. Our goal now is to show that $Y$ contains an $\FF_v$-point $(t,s,x,y)$ such that $p_{\I}(t,s) \neq 0 \in \FF_v$. This, in turn, would imply that $\Y(\OO_v)$ contains a point $(t_v,s_v,x_v,y_v)$ such that $\val_vdp_{\I}(t,s) = 1$, since all $\FF_v$-point of $\Y$ are smooth, and the map $\Y \lrar \PP^1_v$ is smooth over $\OO_v$ (see Remark~\ref{r:smooth}). Let us consider the possible cases:
\begin{enumerate}[(1)]
\item
If $\cA = \emptyset$ and $\cB = \I$ then $p_{\I}(t,s) = p_{\B}(t,s)$ and since $v | a$ any $\FF_v$-point $(t,s,x,y)$ of $\Y$ will necessarily satisfy $p_{\B}(t,s) \neq 0$.
\item
If $\cA = \I$ and $\cB = \emptyset$ then $b$ must be a square at $\FF_v$ (since $\Y$ has an $\FF_v$-point) and $\Y_{(t:s)}(\FF_v) \neq \emptyset$ for any $(t:s) \in \PP^1(\FF_v)$. In this case we just need to show that there exist $t,s \in \FF_v$ such that $p_{\I}(t,s) \neq 0$, which follows from the fact that $|v| \geq |\I|$.
\item
If $\cA \neq \emptyset$ and $\cB \neq \emptyset$ then necessarily $|\I| = 4$ and $|\cA| = |\cB| = 2$. Let us consider two sub-cases. If $p_{\cB}$ is separable, then the curve $C$ is a smooth conic, and hence has at least $|v| + 1 > 6$ points. It then must have point $(t,s,u) \in C(\FF_v)$ such that $p_{\I}(t,s) \neq 0$ (and automatically $u \neq 0$). If $p_{\cB}$ is not separable then the value $bp_{\cB}(t,s)$ is either always a square or always a non-square. Since $\Y$ has an $\FF_v$-point it follows that $bp_{\cB}(t,s)$ is always a square, and so $\Y_{(t:s)}(\FF_v) \neq \emptyset$ for every $(t:s) \in \PP^1(\FF_v)$. In this case we just need to show that there exist $t,s \in \FF_v$ such that $p_{\I}(t,s) \neq 0$, which follows from the fact that $|v| \geq |\I|$.
\end{enumerate}
It is left to show that if $v \notin S_0$ is such that $|v| \leq |\I|$ then $\Y(\OO_v)$ contains a point $(t_v,s_v,x_v,y_v)$ such that $\val_vdp_{\I}(t,s) = 1$. Since $|\I| \leq 4$ and $2 \in S_0$ the only possible prime is $v = 3$, and we may assume that $|\I| = 4$. By our assumption $3$ does not divide $d$. Let us now consider two possible cases. If $p_{\I}$ does not reduce to a multiple of $ts(t-s)(t+s)$ mod $3$ then there exist $\ovl{t}_3,\ovl{s}_3 \in \FF_3$ such that $p_{\I}(\ovl{t}_3,\ovl{s}_3) \neq 0$. We may consequently find $t_3,s_3 \in \OO_3$ such that $\val_vdp_{\I}(t_3,s_3) = 0$, in which case $\left<ap_{\cA}(t,s),bp_{\cB}(t,s)\right> = 0$ and $\Y_{(t_3:s_3)}(\OO_v) \neq \emptyset$. Now assume that $p_{\I}$ does reduce to a multiple of $ts(t-s)(t+s)$. Checking all possibilities it is not hard to verify that in this case there must be $\ovl{t}_3,\ovl{s}_3 \in \FF_3$ such that, say, $p_{\cA}(\ovl{t}_3,\ovl{s}_3) = 0$ while $bp_{\B}(\ovl{t}_3,\ovl{s}_3) = 1$ (this is essentially because the cross-ratio of the four distinct points in $\PP^1(\FF_3)$ is $-1$). Since $p_{\cB}$ is separable (being a divisor of $p_{\I}$) we may find $t_3,s_3 \in \OO_3$, reducing to $\ovl{t}_3,\ovl{s}_3$ mod $3$, and such that $\val_3p_{\cB}(t_3,s_3) = 1$. We now have that $bp_{\cB}(t_3,s_3) \in (\OO^*_3)^2$ and hence $\Y_{(t_3:s_3)}(\OO_3) \neq \emptyset$.
\end{proof}

The following sections are dedicated to the proof of Theorem~\ref{t:main-uncond}.

\subsection{An outline of the proof}
Our goal in this section is to give an outline of the suggested adaptation of Swinnerton-Dyer's descent-fibration method in order to prove Theorem~\ref{t:main-uncond}. For technical reasons it will be convenient to work most of the time on $\X$, rather than $\Y$. We note of course that an $S_0$-integral point on $\X$ yields one on $\Y$. In fact, finding an $S_0$-integral point on $\X$ is an equivalent problem, see Remark~\ref{r:X}.

\begin{enumerate}[(1)]
\item
Given a set of places $T$ containing $S_0$ (as well as other places of special behavior), we will use the term \textbf{partial adelic point} over $T$ to mean a collection of local points of the form
$$ (P_v)_{v \in T} \in \prod_{v \in S_0} \X(\OO_v) \times \prod_{v \in T \bksl S_0} X(k_v) .$$
We will typically denote a partial adelic point indexed on $T$ by $P_T$. The first step in the proof consists in identifying a particular type of partial adelic points whose values at the bad primes is well-behaved (see Definition~\ref{d:suitable}), and then showing that, under the assumptions of Theorem~\ref{t:main-uncond} such suitable adelic points are in sufficient supply. 
\item
Given a suitable adelic point $P_T = (P_v)_{v \in T}$ as above, we begin by finding an $S_0$-integral point $(t_0,s_0)$ on $\A^2 \setminus \{(0,0)\}$ which is sufficiently close to the image of $P_v$ for $v \in T$ and such that the fiber $\X_{(t_0,s_0)} \cong \Y_{(t_0:s_0)}$ contains an $S_0$-integral \textbf{adelic point}. As in some of the other applications of the descent-fibration method, it is convenient to define the notion of a \textbf{$T$-admissible pair} to be a pair of $S$-integral coordinates $(t_0,s_0)$ satisfying the above conditions and such that for each $i \in \I$, the element $p_i(t_0,s_0)$ is prime in $\ZZ_S$, i.e., a product of a prime number and an $S$-unit. We will denote by $u_i$ the place corresponding to the unique prime outside $S$ dividing $p_i(t_0,s_0)$ (see Definition~\ref{d:admissible} for the precise formulation). 

In general, this step of the argument, which is taken up in \S\ref{s:fibration}, is very similar to classical applications of the fibration method for conic bundles, and requires almost no modification to accommodate the integral points setting. To obtain unconditional results, we replace Schinzel's hypothesis by a theorem of Green-Tao-Ziegler (see~\S\ref{ss:gtz}), which is applicable since we have assumed that the bad fibers of the pencil lie above rational points. We note that these issues are not especially related to integral points: the case of the fibration method where all the bad fibers are defined over $\QQ$ can be done unconditionally in the context of rational points as well, see~\cite[\S 9]{HW}.  
\item
Given a $T$-admissible pair $(t,s)$ we may now apply the machinery of \S\ref{ss:selmer} to study the fiber
$$ \X_{(t,s)}: ap_{\cA}(t,s)x^2 + bp_{\cB}(t,s)y^2 = 1 $$
as a torsor over the group scheme
$$ \T_{(t,s)}: x^2 + dp_{\I}(t,s)y^2 = 1 .$$
We denote $T(t,s) = T \cup \{u_i\}_{i \in I}$, and note that the fiber $\X_{(t,s)}$ has good reduction outside $T(t,s)$. Our goal is to be in a position to apply Corollary~\ref{c:final-step} to $\X_{(t,s)}$. For this, we need to know that the dual Selmer group of $\T_{(t,s)}$, as defined in \S\ref{ss:selmer}, is as small as possible, i.e., generated by the class $[-dp_{\I}(t,s)]$. 
\item
Having control over the dual Selmer group of $\T_{(t,s)}$ is a bit tricky in our case, since this group depends on certain features of $t$ and $s$ that cannot be predicted while generating $(t,s)$ using the fibration method together with the Green-Tao-Ziegler theorem. 
To accommodate for this issue, we consider suitable variants of the Selmer and dual Selmer groups, callerd the strict Selmer group and the weak dual Selmer group (see~\S\ref{ss:strict-selmer}). These variants are constructed in such a way that the weak dual Selmer group always contains the dual Selmer group (while the strict Selmer group is always contained in the Selmer group). The key property of interest these variants share is that when applied to $\T_{(t,s)}$ for a $T$-admissible pair $(t,s)$, they turn out to depend only on the associated partial adelic point $P_T$. 
 
To better exploit this observation it will be convenient to write the strict Selmer group and weak dual Selmer group of $\T_{(t,s)}$ in a way that does not depend explicitly on the $T$-admissible pair $(t,s)$. This is part of a general technique which appears in many applications of the descent-fibration method. In our current adaptation we do this by considering certain subgroups $J^T,J_T \subseteq \G$, which depends only on the choice of $T$. For any suitable partial adelic point $P_T$ and a $T$-admissible pair $(t,s)$, evaluation at $(t,s)$ induces isomorphisms $J_T \cong I_{T(t,s)}$ and $J^T \cong I^{T(t,s)}$. We may then identify the strict Selmer group as a subgroup $\Q_{P_T} \subseteq J_T$ and the weak dual Selmer group as a subgroup $\what{\Q}_{P_T} \subseteq J^T$. \S\ref{s:selmer-admis} is devoted to showing that these subgroups indeed depend only on $P_T$. Having established this, the goal of this step in our adaptation of the descent-fibration method can be phrase as follows: we wish to find a subset $T$ containing $S_0$ and a suitable partial adelic point $P_T$ such that the weak dual Selmer group $\what{\Q}_{P_T} \subseteq J^T$ is generated by the class $[-d][p_\I] \in J^T$.
\item
Starting from a given $T$ and $P_T = (P_v)_{v \in T}$ such that the associated weak dual Selmer group $\what{\Q}_{P_T} \subseteq J^T$ contains a non-zero element $x = [c][p_{\J}] \neq [-d][p_\I]$, we wish to find a place $w \notin T$ and a local point $P_w \in \X(\ZZ_w)$ such that the weak dual Selmer group associated to the suitable adelic point $P_{T_w} = (P_v)_{v \in (T \cup \{w\})}$ is contained in $\what{\Q}_{P_T}$ and does not contain $x$. We note that by Remark~\ref{r:ranks-2} the ranks of $\Q_{P_T}$ and $\what{\Q}_{P_T}$ are related by the formula
$$ \dim_2\Q_{P_T} = \dim_2\what{\Q}_{P_T} + S_{\Split} - 1$$ 
where $S_{\Split} \subseteq S_0$ denotes the subset of those places $v \in S_0$ for which $-dp_\I(t_v,s_v)$ is a square in $k_v$, and so reducing $\what{\Q}_{P_T}$ means reducing $\Q_{P_T}$ at the same time. It turns out to be convenient to make this link explicit. We recall that by our assumptions there are at least two places $v^1_\infty, v^2_\infty$ in $S_{\Split}$ and so in particular $\dim_2\Q_{P_T} > \dim_2\what{\Q}_{P_T}$. We may hence conclude that there exists a $x' = [c'][p_{\J'}] \in \Q_{P_T}$ such that $x'$ does not belong to the group generated by $[a][p_{\cA}]$ and $[d][p_{\I}]$. After possibly replacing $x$ by $[-d][p_\I]x$ and $x'$ by $[d][p_\I]x'$ we may assume that $\J \cup \J' \subsetneq \I$ and choose an element $i_x \in \I$ which is not contained in either $\J$ or $\J'$. To fix ideas assume $i \in \B$. Using Condition (D) above and Chebotarev's density theorem one can show that there exists a place $w$ and a pair $t_w,s_w \in \ZZ_w$ such that $p_{i_x}(t_w,s_w) \in \ZZ_w$ is a uniformizer, $ap_\cA(t_w,s_w) \in \ZZ_w^*$ is a square and $cp_\J(t_w,s_w), c'p_{\J'}(t_w,s_w) \in \ZZ_w^*$ are non-square units. 
This means in particular that $\X_{(t_w,s_w)}(\ZZ_w) \neq \emptyset$, and so we may extend our suitable partial adelic point $P_T$ to a suitable partial adelic point $P_{T_w} = (P_v)_{v \in T_w}$, where $T_w := T \cup \{w\}$. The condition that $cp_\J(t_w,s_w)$ is a non-square at $w$ implies that $x \notin \what{\Q}_{P_{T_w}}$. The heart of the proof, which is carried out in~\S\ref{s:compare}, consists in showing that 
$$ \what{\Q}_{P_{T_w}} \subseteq \what{\Q}_{P_T}, $$
and since $\what{\Q}_{P_{T_w}}$ does not contain $x$ it must be strictly smaller, as desired. This step uses in a crucial way the fact that $c'p_{\J'}(t_w,s_w) \in \ZZ_w^*$ is not a square, via a suitable arithmetic duality argument.
\end{enumerate}

\subsection{A Theorem of Green, Tao and Ziegler}\label{ss:gtz}
In a series of papers by Green-Tao and Green–-Tao-–Ziegler, the finite complexity case of the Hardy-Littlewood conjecture was proven (see~\cite[Corollary 1.9]{GT10} for an exact statement). The theorem below is a fairly straightforward consequence of that result, and will be used in this paper in place of Schinzel's hypothesis.
 
\begin{thm}[Green, Tao, Ziegler]\label{t:gtz}
Let $S$ be a finite set of places of $\QQ$ containing the real place. Let $r$ be a positive number and let $f_i(t,s)$ for $i=1,..,r$ be pairwise non-proportional linear forms with coefficients in $\ZZ_S$. Assume that for every $v \notin S$ there exists a pair $(t_v,s_v) \in \ZZ_v^2$ such that $\prod_if_i(t_v,s_v) \in \ZZ_v^*$.

Suppose that we are given pairs $(\lambda_v,\mu_v) \in \QQ_v^2$ for every $v \in S$. Then there exist infinitely many pairs $(\lambda,\mu)\in\ZZ_S^2$ such that
\begin{enumerate}[(1)]
\item
$(\lambda,\mu)$ is arbitrarily close to $(\lambda_v,\mu_v)$ in the $v$-adic topology
for $v\in S \setminus \{\infty\}$.
\item
$(\lambda:\mu)$ is arbitrarily close to $(\lambda_\infty:\mu_\infty)$ in the real topology (on $\PP^1(\RR)$) and $\lambda\lambda_v + \mu\mu_v > 0$.
\item
For $i=1,\ldots,r$ we have $f_i(\lam,\mu)=p_i u_i$ such that $u_i\in \ZZ_S^*$ and
$p_1,\ldots, p_r$ are distinct primes outside $S$.
\end{enumerate}
\end{thm}
\begin{proof}
Since $\QQ^2$ is dense inside $\RR^2$ it is enough to prove the claim under the additional assumption that $t_\infty,s_\infty \in \QQ$. Let $T: \QQ^2 \lrar \QQ^2$ be an invertible linear transformation sending $(t_v,s_v)$ to $(1,0)$. Let $S'$ be a finite set of places containing $S$ such that all the coefficients of $T$ are $S'$-integers, the determinant of $T$ is an $S'$-unit and $f_i(t_\infty,s_\infty)$ is an $S'$-unit for every $i=1,...,r$. By our assumptions we may find, for each $v \in S' \setminus S$, a pair $t_v,s_v \in \ZZ_v$ such that $\prod_if_i(t_v,s_v) \in \ZZ_v^*$. Setting $(\lambda_v,\mu_v) = (t_v,s_v)$ for such $v \in S' \setminus S$ it will now suffice to prove the theorem for $S'$ instead of $S$. Since $T$ induces an invertible $\ZZ_{S'}$-linear transformation $\ZZ_{S'}^2 \lrar \ZZ_{S'}^2$ we may perform a variable change and replace each $(t_v,s_v)$ by $(t_v',s_v') = T(t_v,s_v)$ and each $f_i$ by $f_i' = f_i \circ T^{-1}$. In particular, we have $(t_\infty',s_\infty') = (1,0)$. Since $f_i'(1,0) = f_i(t_\infty,s_\infty)$ is an $S'$-unit we may furthermore divide each $f_i'$ by $f_i'(1,0)$, resulting in a linear form $t - e_is$ for some $S'$-integer $e_i$. We are now in a position to apply~\cite[Proposition 1.2]{HSW14}, itself a direct consequence of the Green-Tao-Ziegler theorem.
\end{proof}

\subsection{The vertical Brauer groups}
In this subsection we prove a few useful claims regarding the vertical Brauer groups of $X$ and $Y$. Our main results are Propositions~\ref{p:vertical} and~\ref{p:reduction} below. We begin with a few preliminaries.

\begin{define}
Let $S_{\bad} \subseteq \Om_\QQ \setminus S_0$ denote the set of places outside $S_0$ for which either $v | \Del_{i,j}$ for some $i\neq j$, or $v | d$, or $v$ is the prime $2$, or $\val_v(p_{\I}(t_v,s_v)) > 0$ for every $t_v,s_v \in \ZZ_v$.
\end{define}

\begin{lem}\label{l:fiber}
Let $v \notin S_0 \cup S_{\bad}$ be a place, $t_v,s_v \in \ZZ_v$ elements such that at least one of $t_v,s_v$ is a unit and $i \in \I$ an element such that $\val_v(p_i(t_v,s_v)) > 0$. Then the fiber $\X_{(t_v,s_v)}(\ZZ_v)$ is non-empty if and only if $aD^{\cA}_i$ is a square at $v$.
\end{lem}
\begin{proof}
By our assumptions on $t_v,s_v$ there exists a unit $u \in \ZZ_v^*$ such that $(t_v,s_v) = (ud_i,-uc_i)$ mod the maximal ideal $\m_{v} \subseteq \ZZ_v$. Let us first assume that $i \in \cA$. Then by Definition~\ref{d:D} we have 
$$ aD^{\cA}_i = bp_{\cB}(d_i,-c_i) $$ 
and so $bp_{\cB}(t_v,s_v) = |u|^{|\cB|}aD^{\cA}_i \in \ZZ_v^*$. On the other hand, since $i \in \cA$ we have $ap_{\cA}(t_v,s_v) =0$ mod $m_v$ and by Hensel's lemma we get that $\X(\ZZ_v) \neq 0$ if and only if $bp_{\cB}(t_v,s_v) \in \ZZ_v^*$ is a square. The desired result now follows from the fact that $|\cB|$ is even. The case where $i \notin \cA$ is completely analogous.
\end{proof}

\begin{define}\label{d:Ai}
For each $i \in \I$ let
$$ A_i \x{\df}{=} \left<aD^{\cA}_i,p_i(t,s)\right> \in \Br(k(X)) $$
be the corresponding Brauer element.
\end{define}

\begin{lem}
The class $A_i$ is Definition~\ref{d:Ai} belongs to $\Br(X) \subseteq \Br(k(X))$.
\end{lem}
\begin{proof}
Let $Z_i \subseteq X$ denote the vanishing locus of $p_i$. Since $aD^{\cA}_i$ is a constant we see that $A_i$ is unramified outside $Z_i$ and that its residue at $D$ is given by the image of the class $[aD^{\cA}_i] \in H^1(k,\ZZ/2)$ in $H^1(Z_i,\ZZ/2)$. Recall that $X$ is determined inside $(\A_k^2 \setminus \{(0,0)\}) \times \A^2_k$ by the equation
$$ ap_{\cA}(t,s)x^2 + bp_{\cB}(t,s)y^2 = 1 $$
If $i \in \cB$ then $p_{\cB}$ vanishes on $Z_i$ and $ap_{\cA}|_{Z_i}$ is a square in $k[Z_i]$. On the other hand, in this case $aD^{\cA}_i = ap_{\cA}(d_i,-c_i)$ (see Definition~\ref{d:D}) coincides with the value of $ap_{\cA}$ on a point of $Z_i$ hence $[aD^{\cA}_i]$ vanishes in $H^1(Z_i,\ZZ/2)$. Similarly, if $i \in \cA$ then $bp_{\cB}|_{Z_i}$ is a square in $k[Z_i]$ and $aD^{\cA}_i = a^bp_{\cB}(d_i,-c_i)$ by Definition~\ref{d:D}, and so again $[aD^{\cA}_i]$ vanishes in $H^1(Z_i,\ZZ/2)$.
\end{proof}

\begin{lem}\label{l:good}
Let $v \notin S_0 \cup S_{\bad}$ be a place of $\QQ$. Then for each $i \in \I$ the Brauer element $A_i$ pairs trivially with every point of $\X(\ZZ_v)$.
\end{lem}
\begin{proof}
Let $v \notin S_0 \cup S_{\bad}$ be a place and $P_v = (t_v,s_v,x_v,y_v) \in \X(\ZZ_v)$ a local point. Since the element $aD^{\cA}_i$ is a $v$-unit we see that 
$A_i(P_v)$ is non-trivial if and only if $aD^{\cA}_i$ is a non-square and $p_i(t,s)$ has an odd (and hence positive) valuation. But this is impossible in light of Lemma~\ref{l:fiber} once the fiber $\X_{(t_v,s_v)}$ has an $\ZZ_v$-point.
\end{proof}

We are now ready to prove two key results concerning the vertical Brauer groups of $X$ and $Y$.
\begin{prop}\label{p:vertical}\
\begin{enumerate}[(1)]
\item
The natural map $\Br^{\verti}(Y,p) \lrar \Br^{\verti}(X,\pi)$ is an isomorphism.
\item
The group $\Br^{\verti}(X,\pi)/\Br(k)$ is generated by the classes of the elements $\sum_{i\in \I} \eps_iA_i$ such that $\prod_{i\in \I} \left(aD^{\cA}_i\right)^{\eps_i}$ is a square.
\item
The class of $\sum_{i\in \I} A_i$ in $\Br^{\verti}(X,\pi)/\Br(k)$ is trivial.
\end{enumerate}
\end{prop}
\begin{proof}
Let us begin with Claim (1). Let $A$ be a class in $\Br^{\verti}(X)$ and let $A' \in \Br(k(\PP^1))$ be a class such that $\pi^*A' = A$. Since $q:  X \lrar Y$ is a $\GG_m$-torsor all its fibers are geometrically irreducible and hence for every codimension $1$ point $P \in Y$ the map $q^*:H^1(k(P),\QQ/\ZZ) \lrar H^1(k(X_P),\QQ/\ZZ)$ is injective. It follows that $p^*A' \in \Br(k(Y))$ must be unramified along any codimension 1 point $P \in Y$. In particular, $p^*A' \in \Br^{\verti}(Y)$ and maps to $A$ in $\Br^{\verti}(X)$, which means that the map $\Br^{\verti}(Y,p) \lrar \Br^{\verti}(X,\pi)$ is surjective. To show that it is also injective consider the generic point $Q \in Y$. Since $X$ is a $\GG_m$-torsor over $Y$ it follows that the generic fiber $X_Q$ has a point defined over $k(Q)$. This implies that the map $\Br(k(Q)) \lrar \Br(X_Q)$ is injective and hence the map $\Br^{\verti}(Y,p) \lrar \Br^{\verti}(X,\pi)$ is injective as well. 

Let us now prove Claim (2). 
Recall the projective conic bundle $\ovl{p}: \ovl{Y} \lrar \PP^1_k$. Since elements in $\pi^*\Br(k(\PP^1_k))$ cannot ramify along $\ovl{Y} \setminus Y$ it follows that the map $\Br^{\verti}(\ovl{Y},\ovl{p}) \lrar \Br^{\verti}(Y,p)$ is an isomorphism. Since $\ovl{p}:\ovl{Y} \lrar \PP^1_k$ is a conic bundle the computation of $\Br^{\verti}(Y,p)$ is standard (see~\cite[\S 7]{Sko99}). For each $i \in \I$ let $P_i \in \PP^1$ be the closed point given by $p_i = 0$. Then the bad fibers of $\ovl{Y}$ are exactly the $P_i$'s, and for each $P_i$ the fiber $\ovl{Y}_{P_i}$ splits in the quadratic extension $\QQ\left(\sqrt{aD^{\cA}_i}\right)$ (but not in $\QQ$). It follows that if $B \in \Br(k(\PP^1))$ is a Brauer element then $p_{\ovl{Y}}^*(B) \in \Br(k(\ovl{Y}))$ is unramified along $\ovl{Y}$ if and only if the element $B$ is unramified outside $\{P_i\}_{i \in \I}$ and for each $i$ we have 
$$ \res(B)|_{P_i} = \eps_i[aD^{\cA}_i] \in H^1(\QQ,\ZZ/2) $$ 
with $\eps_i \in \{0,1\}$. By Faddeev's exact sequence such elements exist if and only if
$\prod_i \left(aD^{\cA}_i\right)^{\eps_i}$ is a square. Furthermore, if $\{\eps_i\}_{i \in \I}$ satisfies this condition then it is straightforward to check that the Brauer element
$$ B = \sum_i \eps_i \left<aD^{\cA}_i,p_i\left(t/s,1\right)\right> \in \Br(k(\PP^1)) $$
has the desired residues. Claim (2) now follows directly.

It is left to prove Claim (3). By comparing residues and using Fadeev's exact sequence we may deduce the existence of an element $B \in \Br(k)$ such that
$$ \sum_i \left<aD^{\cA}_i,p_i(t/s,1)\right> = \left<ap_{\cA}(t/s,1),bp_{\cB}(t/s,1)\right> + B \in \Br(k(\PP^1)) .$$
Since $\pi^*\left<ap_{\cA}(t/s,1),bp_{\cB}(t/s,1)\right>  = 0 \in \Br(k(X))$ by equation~\ref{e:conic-2} the desired result follows.
\end{proof}

Let $T$ be a finite set of places containing $S_0$. By a \textbf{partial adelic point} $(P_v)_{v \in T} = (t_v,s_v,x_v,y_v)_{v \in T}$ on $\X$ we will mean a point in the product
$$ (P_v)_{v \in T} \in \prod_{v \in S_0} X(\QQ_v) \times \prod_{v \in T \setminus S_0} \X(\ZZ_v) .$$
We will usually denote a partial adelic point indexed on $T$ by $P_T$.

\begin{prop}\label{p:reduction}
Let $P_T = (t_v,s_v,x_v,y_v)_{v \in T}$ be a partial adelic point on $\X$ which is orthogonal to all Brauer elements of the form $B = \sum_i n_iA_i$ such that $B \in \Br^{\verti}(X,\pi)$. Then there exists an element $q \in \QQ^*$ which is a unit at every place $v \in T$ and such that the partial adelic point $(qt_v,qs_v,q^{-n}x_v,q^{-m}y_v)$ is orthogonal to $A_i$ for every $i \in \I$.
\end{prop}
\begin{proof}
For each $i \in \I$ let us set
$$ b_i = \sum_{v \in T} A_i(P_v) = \sum_{v \in T} \left<aD^{\cA}_i,p_i(t_v,s_v)\right>_v \in \ZZ/2 .$$
By our assumptions and Proposition~\ref{p:vertical} we know that whenever $\{n_i\}_{i \in \I}$ are such that $\prod_i \left(aD^{\cA}_i\right)^{n_i}$ is a square, the corresponding sum $\sum_i n_ib_i \in \ZZ/2$ vanishes. By Chebotarev's theorem there exists a place $w \notin T$ such that $aD^{\cA}_i$ is a square at $w$ if and only if $b_i = 0$. Let $q$ be a prime which is a uniformizer at $w$ (and so, in particular, $q$ is a unit at every place of $T$). By quadratic reciprocity we then have
$$ \sum_{v \in T} \left<aD^{\cA}_i,q\right>_v = \left<aD^{\cA}_i,q\right>_w = b_i $$
and hence by construction the partial adelic point $(qt_v,qs_v,q^{-n}x_v,q^{-m}y_v)_{v \in T}$ is orthogonal to $A_i$ for every $i \in \I$.
\end{proof}

\subsection{Suitable adelic points}
\begin{const}\label{c:S_D}
For each $x = [c][p_{\J}] \in \G_i \setminus \G_D$, let us fix an $i \neq i_x \in \I$ such that $\left[cD^{\J}_{i_x}\right] \notin \left\{1,\left[aD^{\cA}_{i_x}\right]\right\}$. By Chebotarev's theorem we may then find a place $v_x \notin S_0 \cup S_{\bad}$ such that $aD^{\cA}_{i_x}$ is a square at $v_x$ and yet $cD^{\J}_{i_x}$ is a non-square at $v_x$. Let $t_{x},s_{x} \in \ZZ_{v_x}$ be such that $p_{i_x}(t_{x},s_{x})$ is a uniformizer. By Lemma~\ref{l:fiber} the fiber $\X_{(t_{x},s_{x})}$ contains an $\ZZ_{v_x}$-point. Choosing for each $i$ and each $x \in \G_i \setminus \G_D$ a place $v_x$ as above we obtain a finite set of places which we denote $S_D$. 
\end{const}

Let us now fix our set of distinguished primes to be $S = S_0 \cup S_{\bad} \cup S_D$.

\begin{define}\label{d:suitable}
Let $T$ be a finite set of places containing $S$. We will say that a partial adelic point $P_T = (P_v)_{v \in T}$ is \textbf{suitable} if the following conditions hold
\begin{enumerate}[(1)]
\item
$dp_{\I}(t_v,s_v) \neq 0$ for every $v \in T$.
\item
$\val_v(dp_{\I}(t_v,s_v) \leq 1$ for every $v \in T \setminus S_0$.
\item
If $2 \in T \setminus S_0$ then $\val_2(dp_{\I}(t_v,s_v) = 1$.
\item
There exist two places $v^1_\infty,v^2_\infty \in S_0$ such that $-dp_{\I}(t_{v},s_{v}) \in \QQ_{v}^*$ is a square for $v \in \{v^1_\infty,v^2_\infty\}$.
\item
$\sum_v A_i(P_v) = 0$ for every $i \in \I$ (see Definition~\ref{d:Ai}).
\item
For every $v_x \in S_D$ the point $P_v \in \X(\ZZ_v)$ lies above the point $(t_{x},s_{x}) \in \A^2(\ZZ_{v_x})$  of Construction~\ref{c:S_D}.
\end{enumerate}
\end{define}

\begin{define}\label{d:T0}
Let $T$ be a finite set of places containing $S$ and let $(P_v)_{v \in T} = (t_v,s_v,x_v,y_v)_{v \in T}$ be a suitable partial adelic point on $\X$. We will denote by $T_0 \subseteq T$ the set of places which are either in $S_0$, or for which $\val_v(dp_{\I}(t_v,s_v)) = 1$. 
\end{define}

\begin{rem}
Let $T$ be a finite set of places containing $S$ and let $P_T = (t_v,s_v,x_v,y_v)_{v \in T}$ be a suitable partial adelic point on $\X$. By Property (3) of Definition~\ref{d:suitable} we see that $T_0$ always contains the prime $2$.
\end{rem}

\begin{prop}\label{p:suitable}
Under the assumptions of Theorem~\ref{t:main-uncond}, for every finite set of places $T$ containing $S$ there exists a suitable partial adelic point $P_T$ on $\X$ indexed by $T$.
\end{prop}
\begin{proof}
The assumptions of Theorem~\ref{t:main-uncond} provide an adelic point satisfying properties (2)-(4) of Definition~\ref{d:suitable}. A small deformation can be made to guarantee property (1) and Proposition~\ref{p:reduction} can be used to provide (5). Finally, property (6) can be achieved without disrupting (5) in light of Lemma~\ref{l:good}.
\end{proof}

\subsection{Fibers with points everywhere locally}\label{s:fibration}

\begin{define}
Given an element $x \in \QQ^*$ we will denote by $[x]_v \in \QQ_v^*/(\QQ_v^*)^2$ the equivalence class of the image of $x$ in $\QQ_v^*$.
\end{define}

\begin{define}\label{d:admissible}
Given a set of places $T$ containing $S$, and a suitable partial adelic point $P_T = (t_v,s_v,x_v,y_v)_{v \in T}$ on $\X$, we will say that a pair $(t,s) \in \ZZ_{S_0}^2$ is \textbf{$T$-admissible} with respect to $P_T$ if
\begin{enumerate}[(1)]
\item 
$[p_i(t,s)]_v = [p_i(t_v,s_v)]_v \in \QQ_v^*/(\QQ_v^*)^2$ for every $v \in T$.
\item
For each $i \in \I$, $p_i(t,s)$ is a unit outside $T$ except at exactly one place $u_i$ where $\val_{u_i}(p_i(t,s)) = 1$.
\item
The fiber $\X_{(t,s)}$ has an $S_0$-integral adelic point. 
\end{enumerate}
We will often omit an explicit reference to the partial adelic point $P_T$ when it is clear in the context.
\end{define}

\begin{define}\label{d:T}
Given a $T$-admissible pair $(t,s)$ we will denote by $T(t,s) = T \cup \{u_i\}_{i \in \I}$ and by $T_0(t,s) = T_0 \cup \{u_i\}_{i \in \I}$ (see Definition~\ref{d:T0}). We will denote by $\T_{(t,s)}$ the group scheme over $\ZZ_{T_0(t,s)}$ given by the equation
$$ x^2  + dp_{\I}(t,s)y^2 = 1 $$
\end{define}

\begin{rem}\label{r:control}
Keeping the notation of Definition~\ref{d:T}, we note the fiber $\X_{(t,s)}$ can be identified with the scheme $\ZZ_{S_0}$-scheme $\Z^{ap_{\cA}(t,s)}_0$ studied in \S\ref{s:descent}. If we extend our scalars from $\ZZ_{S_0}$ to $\ZZ_{T_0(t,s)}$ then we may consider $\X_{(t,s)}$ as a torsor under the algebraic torus $\T_{(t,s)}$. Since $(P_v)_{v \in T}$ is suitable it follows that the element $-dp_{\I}(t,s)$ satisfies Assumption~\ref{a:base}. Proposition~\ref{p:final-step} now implies that $\X_{(t,s)}$ has an $S_0$-integral point if and only if it has a $T_0(t,s)$-integral point.
\end{rem}

The following proposition ensures the existence of sufficient supply of $T$-admissible pairs. It can be considered as an instance of the \textbf{fibration method}.
\begin{prop}\label{p:fibration}
Let $T$ be a finite set of places containing $S$ and let $P_T = (t_v,s_v,x_v,y_v)_{v \in T}$ be a suitable partial adelic point on $\X$. Then there exists an $T$-admissible pair $(t,s)$ with respect to $P_T$.
\end{prop}
\begin{proof}
Applying the Theorem~\ref{t:gtz} we may find a pair $(t,s) \in \ZZ_{T}^2$ and places $u_i$ for $i \in \I$ satisfying conditions (1)-(2) of Definition~\ref{d:admissible}. By the inverse function theorem the fiber $\X_{(t,s)}$ has an $\ZZ_v$-point for every $v \in T$. It also clearly has an $\ZZ_v$-point for every $v \notin T \cup \{u_i\}_{i \in \I}$. Now fix an $i \in \I$. By quadratic reciprocity and using Property (5) of Definition~\ref{d:suitable} we get that
$$ \left<aD^{\cA}_i,p_i(t,s)\right>_{u_i} = \sum_{v \in T} \left<aD^{\cA}_i,p_i(t_v,s_v)\right>_v = 0 .$$
Since $p_i(t,s)$ is a uniformizer at $u_i$ and $aD^{\cA}_i,$ is a unit at $u_i$ it follows that $aD^{\cA}_i$ is a square at $u_i$ and so the fiber $\X_{(t,s)}$ has an $\ZZ_{u_i}$-point by Lemma~\ref{l:fiber}.
\end{proof}

\subsection{The Selmer groups of admissible fibers}\label{s:selmer-admis}

Let $T$ be a finite set of places of $\QQ$ containing $S$ and let $P_T := (P_v)_{v \in T}$ be a suitable partial adelic point. Given a $T$-admissible pair $(t,s)$, we wish to control the associated Selmer group $\Sel(\T_{(t,s)},T_0(t,s))$ and dual Selmer group $\Sel(\what{\T}_{(t,s)}, T_0(t,s))$ of the group scheme $\T_{(t,s)}$ with respect to the set of places $T_0(t,s)$. 
This turns out to be difficult to achieve directly. The problem is that this group depends on the various Hilbert symbols $\left<p_j(t,s),p_i(t,s)\right>_{u_i}$, which we cannot fully control when generating our $T$-admissible pairs via Proposition~\ref{p:fibration}. To overcome this difficulty we will replace the Selmer group with the \textbf{strict Selmer group} and the dual Selmer group with the \textbf{weak dual Selmer group} (see \S\ref{ss:strict-selmer}). It turns out that these variants are in fact \textbf{independent} of the specific $T$-admissible pair we choose, and only depend on the original suitable adelic point $P_T$ (see Corollary~\ref{c:indep} below).

\begin{define}\label{d:J}
Let $T$ be a finite set of places of $\QQ$ containing $S$. We will denote by $J_T \subseteq \G$ (resp. $J^T \subseteq \G$) the subgroup generated by $I_{T}$ (resp. $I^{T}$) and the symbols $[p_i]$ (see~\S\ref{ss:selmer} and~\S\ref{s:main} for the notation of $I_T$ and $\G$ respectively). For a $T$-admissible pair $(t,s)$, we will denote by $\ev_{(t,s)}: J_{T} \lrar I_{T(t,s)}$ the natural map obtained by sending $[p_i]$ to $[p_i(t,s)]$. We similarly define the evaluation map $\ev^{(t,s)}: J^{T} \lrar I^{T(t,s)}$. By our assumptions on $T$ the maps $\ev_{(t,s)}$ and $\ev^{(t,s)}$ are isomorphisms of $\FF_2$-vector spaces. 
\end{define}

Let is recall the notation used in \S\ref{ss:selmer}. For each place $v$ of $\QQ$, let $V_v$ and $V^v$ denote two copies of $H^1(\QQ_v,\ZZ/2) \cong \QQ_v^*/(\QQ_v^*)^2$, considered as $\FF_2$-vector spaces and let $W^v(t,s) \subseteq V^v$ be the subspace defining the dual Selmer condition at $v$ for the $\ZZ_{T_0(t,s)}$-group scheme $\T_{(t,s)}$. More explicitly, if $v \in T_0(t,s)$ then $W^v(t,s)$ is the subspace generated by the class $[-dp_{\I}(t,s)]$, and if $v \notin T_0(t,s)$ then $W^v(t,s)$ is the image of $\ZZ_v^*/(\ZZ_v^*)^2$ inside $V^v$. In both cases, the subspace $W_v(t,s) \subseteq V_v$ defining the Selmer condition of $\T_{(t,s)}$ at $v$ is given by the orthogonal complement of $W^v(t,s)$ with respect to the Hilbert symbol pairing.

\begin{define}\label{d:Q-hat}
We will denote by $\Q_{P_T} \subseteq J_{T}$ the inverse image of the \textbf{strict Selmer group} $\Sel^{-}(\T_{(t,s)}, T_0(t,s), \{u_i\}_{i \in \I})$ via $\ev_{(t,s)}$ and by $\what{\Q}_{P_T} \subseteq J^{T}$ the inverse image of the \textbf{weak dual Selmer group} $\Sel^{+}(\what{\T}_{(t,s)}, T_0(t,s),\{u_i\}_{i \in \I})$ via $\ev^{(t,s)}$ (see \S\ref{ss:strict-selmer}). We note that 
$$ \Q_{P_T} \cong \Sel^{-}(\T_{(t,s)}, T_0(t,s),\{u_i\}_{i \in \I}) $$ 
and 
$$ \what{\Q}_{P_T} \cong \Sel^{+}(\what{\T}_{(t,s)}, T_0(t,s),\{u_i\}_{i \in \I}) .$$
\end{define}

\begin{rem}
The notation in Definition~\ref{d:Q-hat} may appear abusive, since $\Q_{P_T}$ and $\what{\Q}_{P_T}$ depend, a-priori, on the choice of $T$-admissible pair $(t,s)$, and not just on $P_T$. However, we will below that the dependence on $(t,s)$ is vacuous, and the subgroup $\Q_{P_t}$ and $\what{\Q}_{P_t}$ depend, in fact, only on the suitable partial adelic point $P_T$ (see Corollary~\ref{c:indep}). 
\end{rem}

\begin{rem}\label{r:not-depend}
By the definition of $T$-admissible pair we have $[p_i(t,s)]_v = [p_i(t_v,s_v)]_v \in \QQ_v^*/(\QQ_v^*)^2$ for every $v \in T$, and so for such $v$ the groups $W_v(t,s)$ and $W^v(t,s)$ do not depend on the choice of the $T$-admissible pair $(t,s)$. 
\end{rem}

The next lemma identifies the exact dependency of $W_v(t,s)$ and $W^v(t,s)$ on $t,s$ for places of a certain type (that which includes in particular, the places $\{u_i\}_{i \in \I}$ associated to $(t,s)$).

\begin{lem}\label{l:local-selmer}
Let $T$ be a finite set of places of $\QQ$ containing $S$, let $P_T = \{P_v\}_{v \in T}$ be a suitable partial adelic point and let $(t,s)$ be a $T$-admissible pair with respect to $P_T$. Let $i \in \I$ be an element and let $v \in T_0(t,s)$, $v \notin S_0 \cup S_{\bad}$ be an odd place such that $(t,s) = (rd_i,-rc_i)$ mod the maximal ideal $\m_{v} \subseteq \ZZ_v$ for some $r \in \ZZ_v^*$. Then for $x = [c][p_{\J}] \in J_T$ we have 
$$ [cu^{|\J|}D^{\J}_i]_v = \left\{\begin{matrix} [\ev_{(t,s)}(x)]_v & i \notin \J \\  [\ev_{(t,s)}(x[d][p_{\I}])]_v & i \in \J \end{matrix}\right. $$ 
and $\ev_{(t,s)}(x) \in W_v(t,s)$ if and only if $[cr^{|\J|}D^{\J}_i]_v = 0$. Similarly, for $x = [c][p_{\J}] \in J^T$ we have
$$ [cr^{|\J|}\what{D}^{\J}_i]_v = \left\{\begin{matrix} [\ev_{(t,s)}(x)]_v & i \notin \J \\  [\ev_{(t,s)}(x[-d][p_{\I}])]_v & i \in \J \end{matrix}\right. $$ 
and $\ev^{(t,s)}(x) \in W^v(t,s)$ if and only if $[cr^{|\J|}\what{D}^{\J}_i]_v = 0$.
\end{lem}
\begin{proof}
Since $v \notin S_0 \cup S_{\bad}$ we have 
$$  \val_v(p_j(t,s)) = \val_v p_j(d_i,-c_i) = \val_vd_{i,j} = 0 $$
for every $j \neq i$, while by our assumptions we have $\val_v(p_i(t,s)) > 0$. Since $P_T$ is suitable and $(t,s)$ is $T$-admissible we have $\val_v dp_{\I}(t,s) \leq 1$ and hence $\val_v p_i(t,s) = \val_vdp_{\I}(t,s) = 1$. It follows that the $v$-valuation of $\ev_{(t,s)}(x)$ is $1$ if $i \in \J$ and $0$ otherwise. Now $W^v(t,s) \subseteq V^v$ is the subgroup generated by $[-dp_{\I}(t,s)]_v$ and since $v$ is odd we get that $W_v(t,s) \subseteq V_v$ is the subgroup generated by $[dp_{\I}(t,s)]_v$. The claim now follows directly from the definitions of $D^{\J}_i$ and $\what{D}^{\J}_i$ (see Definition~\ref{d:D}).
\end{proof}

The following lemma gives a convenient characterization of elements in the weak dual Selmer group $\what{\Q}_{P_T}$.
\begin{lem}\label{l:explicit-2}
Let $T$ be a finite set of places containing $S$ and $(t,s)$ a $T$-admissible pair with associated places $\{u_i\}_{i \in \I}$. For each $i \in \I$, let $r_i \in \ZZ_{u_i}^*$ be such that $(t,s)$ reduces to $(r_id_i,-r_ic_i)$ mod $u_i$ (such an $r_i$ exists since $p_i(t,s) = c_it + d_is$ is a uniformizer in $\ZZ_v$). Let $x = [c][p_{\J}] \in J_T$ be an element. Then $x$ belongs to $\what{\Q}_{P_T}$ if and only if the following holds:
\begin{enumerate}[(1)]
\item
For every $v \in T$ the element $\ev^{(t,s)}(x)$ belongs to $W^{v}(t,s)$.
\item
There exists a $b \in \ZZ/2$ such that for every $i \in \J$ we have
$$ \left<p_i(t,s), cr_i^{|\J|}\what{D}^{\J}_i\right>_{u_i} = b $$ 
\end{enumerate}
\end{lem}
\begin{proof}
Let $T' = T(t,s)$ and let $\alp = \ev^{(t,s)}(x) \in I_{T'}$. Recall that $T' \setminus T = \{u_i\}_{i \in \I}$. By Remark~\ref{r:2-descent-strict} we see that $\alp$ belongs to 
$$\Sel^{+}(\what{\T}_{(t,s)}, T_0(t,s), \{u_i\}_{i \in \I})$$ 
if and only if for every $(x_v)_{v \in T'} \in \oplus_{v \in T'}W_v(t,s)$ such that $\sum_{v \in T' \setminus T} \val_v x_v$ is even one has
$$ \sum_{v \in T'} \left<x_v,\alp\right>_v = 0 .$$
Since each $W_v(t,s)$ is at most $1$-dimensional, and since for every $v \in T' \setminus T$ the generator of $W_v(t,s)$ has odd valuation, a simple linear algebra argument shows that this condition can be equivalently stated as follows. For each $v \in T'$ let $x_v \in W_v(t,s)$ be a generator. Then $\alp$ belongs to the weak dual Selmer group if and only if there exists a $b \in \ZZ/2$ such that
$$ \left<x_v, \alp\right>_v = \left\{\begin{matrix} 0 & v \in T \\ b & v \in T' \setminus T \\ \end{matrix}\right. .$$
By construction we may take $x_v = [dp_{\I}(t,s)]_v$ for every $v \in T' \setminus T$. Now recall that $\val_{u_i}(p_i(t,s)) = \val_{u_i}(dp_{\I}(t,s)) = 1$ and $\val_{u^i}(\alp) = \val_{u^i}(cp_{\J}(t,s))$ is $1$ if $i \in \J$ and $0$ otherwise. It follows that for every $i \in \J^c$ we have
$\left<dp_{\I}(t,s), \alp\right>_{u_i} = \left<p_i(t,s), \alp\right>_{u_i}$ and for every $i \in \J$ we have 
$$ \left<dp_{\I}(t,s), \alp)\right>_{u_i} = \left<dp_{\I}(t,s), \alp(-dp_{\I}(t,s))\right>_{u_i} = \left<dp_{\I}(t,s), \ev^{(t,s)}(x[-d][p_{\I}])\right>_{u_i} .$$ 
The desired result is now an immediate consequence.
\end{proof}

\begin{cor}\label{c:indep}
The subgroups $\Q_{P_T} \subseteq J_T$ and $\what{\Q}_{P_t} \subseteq J^T$ do not depend on $(t,s)$, but only on $P_T$. In particular, the strong Selmer group and weak dual Selmer group of $\T_{(t,s)}$ is essentially independent of $(t,s)$, as long as $(t,s)$ is a $T$-admissible pair with respect to $P_T$.
\end{cor}
\begin{proof}
We will prove the claim for $\Q_{P_T}$. The proof for $\what{\Q}_{P_T}$ is completely analogous. Let $(t_0,s_0)$ and $(t_1,s_1)$ be two $T$-admissible pairs with respect to $P_T$ with associated places $\{u^0_i\}_{i \in \I}$ and $\{u^1_i\}_{i \in \I}$ respectively. For $i \in \I$ let $r^0_i,r^1_i \in \ZZ_v^*$ be such that such that $(t_0,s_0)$ reduces to $(r^0_id_i,-r^0_ic_i)$ mod $u^0_i$, and similarly for $r^1_i$. In light of Lemma~\ref{l:explicit-2} and Remark~\ref{r:not-depend} it will suffice to show that a given element $x = [c][p_\J] \in J_T$ satisfies Condition (2) of Lemma~\ref{l:explicit-2} with respect to $(t_0,s_0)$ if and only if it satisfies it with respect to $(t_1,s_1)$. By Lemma~\ref{l:local-selmer} this assertion holds as soon as $|\J|$ as even. It will hence suffice to prove the claim for a single $x$ with $|\J|$ odd. In particular, we may assume that $\J = \{j\}$ for some $j \in \I$ and $c=1$. Unwinding the definitions, what we need to show in this case is simply that the element $\left<p_i(t_0,s_0), r^0_iD^j_i\right>_{u^0_i} + \left<p_i(t_1,s_1), r^1_iD^j_i\right>_{u^1_i} \in \ZZ/2$ is independent of $i$. We first claim that
\begin{equation}\label{e:ij}
\left<p_i(t_0,s_0), r^0_iD^j_i\right>_{u^0_i} + \left<p_i(t_1,s_1), r^1_iD^j_i\right>_{u^1_i} = \left<r^0_jD^i_j, p_j(t_0,s_0)\right>_{u^0_j} + \left<r^1_jD^i_j, p_j(t_1,s_1)\right>_{u^1_j}.
\end{equation}
Since~\eqref{e:ij} clearly holds when $i=j$ we may assume that $i \neq j$. Quadratic reciprocity now gives
$$ \left<p_i(t_0,s_0), r^0_iD^j_i\right>_{u^0_i} + \left<p_i(t_1,s_1), r^1_iD^j_i\right>_{u^1_i} =  \left<p_i(t_0,s_0), p_j(t_0,s_0)\right>_{u^0_i} + \left<p_i(t_1,s_1), p_j(t_1,s_1)\right>_{u^1_i} = $$
$$ \left<p_i(t_0,s_0), p_j(t_0,s_0)\right>_{u^0_j} + \left<p_i(t_1,s_1), p_j(t_1,s_1)\right>_{u^1_j} + \sum_{v \in T_0} \left[\left<p_i(t_0,s_0),p_j(t_0,s_0)\right>_v + \left<p_i(t_1,s_1),p_j(t_1,s_1)\right>_v\right] =  $$
$$ \left<p_i(t_0,s_0), p_j(t_0,s_0)\right>_{u^0_j} + \left<p_i(t_1,s_1), p_j(t_1,s_1)\right>_{u^1_j} = \left<r^0_jD^i_j, p_j(t_0,s_0)\right>_{u^0_j} + \left<r^1_jD^i_j, p_j(t_1,s_1)\right>_{u^1_j} $$ 
where the last equality follows from Lemma~\ref{l:local-selmer}. It is hence left to show that the last expression does not depend on $i$. Using again quadratic reciprocity we get that
$$ \left<D^i_j, p_j(t_0,s_0)\right>_{u^0_j} = \sum_{v \in T_0}\left<D^i_j,p_j(t_v,s_v)\right> = \left<D^i_j, p_j(t_1,s_1)\right>_{u^1_j} $$ 
and hence 
$$ \left<r^0_jD^i_j, p_j(t_0,s_0)\right>_{u^0_j} + \left<r^1_jD^i_j, p_j(t_1,s_1)\right>_{u^1_j} = \left<r^0_j, p_j(t_0,s_0)\right>_{u^0_j} + \left<r^1_j, p_j(t_1,s_1)\right>_{u^1_j} $$
is independent of $i$, as desired.
\end{proof}

We finish this section by noting another direct corollary of the above, which will be useful in the next section.
\begin{cor}\label{c:special}
Let $T$ be a finite set of places containing $S$ and let $(t,s)$ be a $T$-admissible pair. If an element $x = [c][p_{\J}] \in \Q_{(t,s)}$ belongs to $\G_i$ for some $i$ then $x$ belongs to $\G_D$.
\end{cor}
\begin{proof}
By the definition of $\Q_{(t,s)}$ and the strict Selmer group we have that $|\J|$ is even. In light of Lemma~\ref{l:local-selmer} we see that if $x \in \G_i$ is an element which is not in $\G_D$ then $\ev_{(t,s)}(x)$ will fail to satisfy the Selmer condition at the place $v_x$ constructed in Construction~\ref{c:S_D}. 
\end{proof}

\subsection{Comparing Selmer groups of nearby fibers}\label{s:compare}
Let $T$ be a finite set of places, let $P_T = (P_v)_{v \in T}$ be a suitable partial adelic point and let $(t_0,s_0)$ be a $T$-admissible pair with associated places $\{u^0_i\}_{i \in \I}$. In~\S\ref{s:selmer-admis} we saw that the strict Selmer group and weak dual Selmer group of $\T_{(t_0,s_0)}$ depend only on $P_T$, and can be expressed as certain subgroups $\Q_{P_T} \subseteq J_T$ and $\what{\Q}_{P_T} \subseteq J^T$. Our goal in this section is to understand what happens when one adds to $T$ one additional place $w$. Let $T_w = T \cup \{w\}$ and let $P_{T_w} = (P_v)_{v \in T_w} = (x_v,y_v,t_v,s_v)_{v \in T_w}$ be an extension of $P_T$ to $T_w$. We will make the following assumption:

\begin{assume}\label{a:w}
There exists an $i_w \in \I$ for which $p_{i_w}(t_w,s_w)$ is a uniformizer. 
\end{assume}

\begin{rem}\label{r:w}
Since $w \notin S_0$ it follows from Assumption~\ref{a:w} that $p_{i}(t_w,s_w)$ is a unit in $\ZZ_w$ for every $i \neq i_w$.
\end{rem}

Let $(t_1,s_1)$ be a $T_w$-admissible pair with respect to $P_{T_w}$ with associated places $\{u^1_i\}_{i \in \I}$ for $i \in \I$. Our goal in this subsection is to be able to compare the weak dual Selmer groups $\T_{(t_0,s_0)}$ and $\T_{(t_1,s_1)}$, i.e., to compare the subgroups $\what{\Q}_{P_T}$ and $\what{\Q}_{P_{T_w}}$ of $J^T$. This will require some understanding of difference between the strict Selmer groups $\Q_{P_T}$ and $\Q_{P_{T_w}}$ as well. A comparison of the Selmer conditions for $\T_{(t_0,s_0)}$ and $\T_{(t_1,s_1)}$ is described in Proposition~\ref{p:transfer-0} below, leading to our main interest which is Proposition~\ref{p:mazur-formula}, giving sufficient conditions for the weak dual Selmer group of $\T_{(t_1,s_1)}$ to be strictly smaller than that of $\T_{(t_0,s_0)}$, i.e., for $\what{\Q}_{P_{T_w}}$ to be properly contained in $\what{\Q}_{P_T}$. We begin with a few preliminary lemmas.

\begin{lem}\label{l:b}
There exists an element $b \in \ZZ/2$ such that for every $i \neq j \in \I$ with $i \neq i_w$ we have
$$ \left<p_i(t_0,s_0),p_j(t_0,s_0)\right>_{u^0_i} + \left<p_i(t_1,s_1),p_j(t_1,s_1)\right>_{u^1_i} = b $$
\end{lem}
\begin{proof}
If $|\I| = 2$ then the claim holds trivially. We may hence assume that $|\I| > 2$. Let us set 
$$ b_{i,j} \x{\df}{=} \left<p_i(t_0,s_0),p_j(t_0,s_0)\right>_{u^0_i} + \left<p_i(t_1,s_1),p_j(t_1,s_1)\right>_{u^1_i} $$
We first observe that if $i,j \neq i_w$ then by quadratic reciprocity
$$ \left<p_i(t_0,s_0),p_j(t_0,s_0)\right>_{u^0_i} + \left<p_i(t_0,s_0),p_j(t_0,s_0)\right>_{u^0_j} = $$ 
$$ \sum_{v \in T}\left<p_i(t_0,s_0),p_j(t_0,s_0)\right>_v = \sum_{v \in T}\left<p_i(t_1,s_1),p_j(t_1,s_1)\right>_v = $$
$$ \left<p_i(t_1,s_1),p_j(t_1,s_1)\right>_{u^1_i} + \left<p_i(t_1,s_1),p_j(t_1,s_1)\right>_{u^1_j} $$
and so $b_{j,i} = b_{i,j}$ whenever $i,j \neq i_w$. We now claim that if $i,j,k \in \I$ are three distinct elements such that $i \neq i_w$ then $b_{i,j} = b_{i,k}$. To see this observe that by Lemma~\ref{l:local-selmer} we have
$$ \left<p_i(t_0,s_0), p_{\{j,k\}}(t_0,s_0)\right>_{u^0_i} = \left<p_i(t_0,s_0), D^{\{j,k\}}_i\right>_{u^0_i} = $$
$$ \sum_{v \in T} \left<p_i(t_0,s_0), D^{\{j,k\}}_i\right>_v = \sum_{v \in T} \left<p_i(t_1,s_1), D^{\{j,k\}}_i\right>_v = $$
$$ \left<p_i(t_1,s_1),D^{\{j,k\}}_i\right>_{u^1_i} = \left<p_i(t_1,s_1), p_{\{j,k\}}(t_1,s_1)\right>_{u^1_i} .$$
Applying the equalities $b_{j,i} = b_{i,j}$ and $b_{i,j}=b_{i,k}$ for all distinct triplets $i,j,k \in \I \setminus \{i_w\}$ we may deduce the existence of a $b \in \ZZ/2$ such that $b_{i,j} = b$ for every $i \neq j \in \I \setminus \{i_w\}$. Finally, to see that $b_{i, i_w} = b$ for every $i \neq i_w$ we may choose a $j \in \I$ such that $j \neq i,i_w$ (which is possible since we assumed $|\I| > 2$) and use the equality $b_{i, i_w} = b_{i,j} = b$.
\end{proof}

\begin{cor}\label{c:b}
Let $b \in \ZZ/2$ be the element constructed in Lemma~\ref{l:b}, let $x = [c][p_{\J}] \in J^T$ be an element and let $i \in \J^c$. When $i=i_w$ we assume in addition that $cp_{\J}(t_1,s_1)$ is a square at $w$. Then
$$ \left<p_i(t_0,s_0), cp_{\J}(t_0,s_0)\right>_{u^0_i} + \left<p_i(t_1,s_1),cp_{\J}(t_1,s_1)\right>_{u^1_i} = b\cdot |\J| \in \ZZ/2 . $$
\end{cor}
\begin{proof}
Let us first assume that $i \neq i_w$. Quadratic reciprocity implies that
$\left<p_i(t_0,s_0), c\right>_{u^0_{i}} = \left<p_i(t_1,s_1),c\right>_{u^1_{i}}$ and hence by Lemma~\ref{l:b} we have 
$$ \left<p_i(t_0,s_0), cp_{\J}(t_0,s_0)\right>_{u^0_i} + \left<p_i(t_1,s_1), cp_{\J}(t_1,s_1)\right>_{u^1_i} = $$
$$ \left<p_i(t_0,s_0), p_{\J}(t_0,s_0)\right>_{u^0_i} + \left<p_i(t_1,s_1), p_{\J}(t_1,s_1)\right>_{u^1_i} = $$
$$ \sum_{j \in \J} \left<p_i(t_0,s_0), p_j(t_0,s_0)\right>_{u^0_i} + 
\sum_{j \in \J} \left<p_i(t_1,s_1), p_j(t_1,s_1)\right>_{u^1_i} = b|\J| .$$
Now let $i=i_w$ and assume that $cp_{\J}(t_1,s_1)$ is a square at $w$. For every $v \in T$ we have
$ \left<p_{i_w}(t_0,s_0), cp_{\J}(t_0,s_0)\right>_v = \left<p_{i_w}(t_1,s_1), cp_{\J}(t_1,s_1)\right>_v $
and so by using quadratic reciprocity and the fact that $cp_{\J}(t_1,s_1)$ is a square at $w$ we may conclude that
$$  \left<p_{i_w}(t_0,s_0), cp_{\J}(t_0,s_0)\right>_{u^0_{i_w}} + \left<p_{i_w}(t_1,s_1), cp_{\J}(t_1,s_1)\right>_{u^1_{i_w}} = $$
$$ \sum_{j \in \J}\left[\left<p_{i_w}(t_0,s_0), cp_{\J}(t_0,s_0)\right>_{u^0_j} + \left<p_{i_w}(t_1,s_1), cp_{\J}(t_1,s_1)\right>_{u^1_j} \right] = $$
$$ \sum_{j \in \J}\left[\left<p_{i_w}(t_0,s_0), p_j(t_0,s_0)\right>_{u^0_j} + \left<p_{i_w}(t_1,s_1), p_{j}(t_1,s_1)\right>_{u^1_j}\right] = b|\J| $$

\end{proof}

Our first comparison result relates the Selmer conditions of $\T_{(t_0,s_0)}$ and $\T_{(t_1,s_1)}$ at the places of $T$ and the places $u^0_i,u^1_i$ for $i \neq i_w$.
\begin{prop}\label{p:transfer-0}
Let $x = [c][p_{\J}] \in J_T$ be an element such that $|\J|$ is even. Then 
\begin{enumerate}[(1)]
\item
For every $v \in T$ we have $\ev_{(t_0,s_0)}(x) \in W_v(t_0,s_0) \Leftrightarrow \ev_{(t_1,s_1)}(x) \in W_v(t_1,s_1)$.
\item
For every $i \neq i_w \in \I$ we have $\ev_{(t_0,s_0)}(x) \in W_{u^i_0}(t_0,s_0) \Leftrightarrow \ev_{(t_1,s_1)}(x) \in W_{u^i_1}(t_1,s_1)$.
\end{enumerate}
\end{prop}
\begin{proof}
Claim (1) follows from Condition(1) of Definition~\ref{d:admissible}. Now let $i \neq i_w$ be an element of $\I$. Since $|\J|$ is even we may combine Corollary~\ref{c:b} with Lemma~\ref{l:local-selmer} and conclude that
$$ \left<p_i(t_0,s_0), cD^{\J}_i\right>_{u^0_i} + \left<p_i(t_1,s_1), cD^{\J}_i\right>_{u^1_i} = 0 .$$
Since $\val_{u^0_i}p_i(t_0,s_0) = \val_{u^1_i}p_i(t_1,s_1) = 1$ and $cD^{\J}_i$ is an $S$-unit we see that Claim (2) now follows from Lemma~\ref{l:local-selmer}.
\end{proof}

Let us now turn our attention to comparing the weak dual Selmer groups of $\T_{(t_0,s_0)}$ and $\T_{(t_1,s_1)}$. Let $V_w, V^w$ be two copies of $H^1(\QQ_w,\ZZ/2)$. We have natural maps 
$$ \loc_w:J_{T_w} \lrar V_w $$ 
and 
$$ \loc^w:J^{T_w} \lrar V^w $$ 
obtained by composing $\ev_{(t_1,s_1)}$ (resp. $\ev^{(t_1,s_1)}$) with the localization map at $w$. 

\begin{define}\label{d:Q0}
For a set of places $T'$ and a suitable partial adelic point $P_{T'}$ we define the subgroups $\Q_{P_{T'}}^0 \subseteq \Q_{P_{T'}}$ and $\what{\Q}_{P_{T'}}^0 \subseteq \what{\Q}_{P_{T'}}$ to be the subgroups consisting of those elements 
$[c][p_{\J}]$ such that $i_w \notin \J$. 
\end{define}

\begin{rem}\label{r:Q_0}
In the setting of Definition~\ref{d:Q0} the $\FF_2$-vector space $\Q_{P_{T'}}$ is generated by $\Q^0_{P_{T'}}$ and $[d][p_{\I}]$. Since $[d][p_{\I}]$ does not belong to $\Q^0_{P_{T'}}$ we obtain a direct sum decomposition $\Q_{P_{T'}} \cong \Q^0_{P_{T'}} \oplus \FF_2\langle[d][p_{\I}]\rangle$. In a similar way, we have a direct sum decomposition $\what{\Q}_{P_{T'}} \cong \what{\Q}^0_{P_{T'}} \oplus \FF_2\langle[-d][p_{\I}]\rangle$.
\end{rem}

\begin{define}\label{d:P}
We denote by $P_0 \subseteq V_w$ be the image of $\Q^0_{P_{T}}$ via $\loc_w$ and by $P^0$ be the image of $\what{\Q}^0_{P_{T}}$ via $\loc^w$. Similarly, we denote by $P_1 \subseteq V_w$ be the image of $\Q^0_{P_{T_w}}$ via $\loc_w$ and by $P^1$ be the image of $\what{\Q}^0_{P_{T_w}}$ via $\loc^w$. 
\end{define}

\begin{rem}\label{r:dim-1}
By Definition~\ref{d:Q-hat} the elements of $\Q_{P_T}$ (resp. $\what{\Q}_{P_T}$) satisfy the Selmer (resp. dual Selmer) condition of $\T_{(t_1,s_1)}$ at $w$ and so $P_1 \subseteq W_w(t_1,s_1)$ (resp. $P^1 \subseteq W^w(t_1,s_1)$). In particular, we have $\dim_2P_1,\dim_2P^1 \leq 1$. On the other hand, by Remark~\ref{r:w} we have that $p_i(t_1,s_1)$ is a unit at $w$ for $i \neq i_w$ and since $w \notin T$ it follows that $P_0,P^0 \subseteq \ZZ_w^*/(\ZZ_w^*)^2$. In particular, $\dim_2P_0,\dim_2P^0 \leq 1$ as well.
\end{rem}

\begin{lem}\label{l:ortho}
The Hilbert pairing of any element in $P_0 \subseteq V_w$ with any element of $P^1 \subseteq V^w$ is trivial.
\end{lem}
\begin{proof}
Let $x = [c][p_{\J}] \in \Q^0_{P_T}$ and $x' = [c'][p_{\J'}]\in \what{\Q}^0_{P_{T_w}}$ be elements.
By Definition~\ref{d:Q0} we have $i_w \notin \J$ and $i_w \notin \J'$. It follows that both $\ev_{(t_1,s_1)}(x)$ and $\ev_{(t_1,s_1)}(x')$ are units at $u^1_{i_w}$, and hence 
$$ \left<\ev_{(t_1,s_1)}(x),\ev^{(t_1,s_1)}(x')\right>_{u^1_w} = 0 .$$
On the other hand, for every $v \in T$ and every $v = u^1_i$ with $i \neq i_w$ we have from Proposition~\ref{p:transfer-0} that $\ev_{(t_1,s_1)}(x) \in W_v(t_1,s_1)$. Now, by Definition~\ref{d:Q-hat} we have $\ev^{(t_1,s_1)}(x') \in W^v(t_1,s_1)$ and since $W_v(t_1,s_1)$ is orthogonal to $W^v(t_1,s_1)$ it follows that
$$ \left<\ev_{(t_1,s_1)}(x),\ev^{(t_1,s_1)}(x')\right>_v = 0 $$
for these $v$'s as well. Since $\ev_{(t_1,s_1)}(x)$ and $\ev^{(t_1,s_1)}(x')$ are units outside $T_w(t_1,s_1)$ it follows from quadratic reciprocity that
$$ \left<\ev_{(t_1,s_1)}(x),\ev^{(t_1,s_1)}(x')\right>_w = \sum_{v \in T \cup \{u^1_i\}_{i \in \I}}\left<\ev_{(t_1,s_1)}(x),\ev^{(t_1,s_1)}(x')\right>_v = 0 $$
as desired.
\end{proof}

We are now in a position to compare the weak dual Selmer groups $\what{\Q}_{P_T}$ and $\what{\Q}_{P_{T_w}}$.

\begin{prop}\label{p:mazur-formula}\
\begin{enumerate}[(1)]
\item
Let $x = [c][p_{\J}] \in \what{\Q}^0_{P_{T_w}}$ be an element such that $\loc^w(x) = 0$. Then $x$ belongs to $\what{\Q}^0_{P_T}$.
\item
If both $P_0$ and $P^0$ are non-zero then $P^1 = \{0\}$ and
$ \what{\Q}^0_{P_{T_w}} \subsetneq \what{\Q}^0_{P_T} $.
\end{enumerate}
\end{prop}
\begin{proof}
Let us first prove (1). Since $x \in \what{\Q}^0_{P_{T_w}}$ we have $i_w \notin \J$ by Definition~\ref{d:Q0}. Since furthermore $\loc^w(x) = 0$ it follows that $c$ is a unit at $w$, i.e., $c \in \ZZ_T^*$. Now since $\ev^{(t_1,s_1)}(x) \in W^v(t_1,s_1)$ for every $v \in T$, Condition (1) of Definition~\ref{d:admissible} implies that $\ev^{(t_0,s_0)}(x) \in W^v(t_0,s_0)$ for every $v \in T$. By Lemma~\ref{l:explicit-2} there exists a $b \in \ZZ/2$ such that for every $i \in \J^c$ we have
$\left<p_i(t_1,s_1), \ev^{(t_1,s_1)}(x)\right>_{u^1_i} = b$ and for every $i \in \J$ we have $\left<p_i(t_1,s_1), \ev^{(t_1,s_1)}(x[-d][p_{\I}])\right>_{u^1_i} = b$. By our assumption we have that $\ev_{(t_1,s_1)}(x)$ is a square at $w$. Applying Lemma~\ref{l:b} and Corollary~\ref{c:b}, we may deduce the existence of a $b' \in \ZZ/2$ such that for every $i \in \J^c$ (including $i=i_w$) we have
$$  \left<p_i(t_0,s_0), \ev^{(t_0,s_0)}(x)\right>_{u^0_i} + \left<p_i(t_1,s_1), \ev^{(t_1,s_1)}(x)\right>_{u^1_i} = b'\cdot|\J| \in \ZZ/2 $$
and for every $i \in \J$ (which means in particular that $i \neq i_w$) we have
$$  \left<p_i(t_0,s_0), \ev^{(t_0,s_0)}(x[-d][p_{\I}])\right>_{u^0_i} + \left<p_i(t_1,s_1), \ev^{(t_1,s_1)}(x[-d][p_{\I}])\right>_{u^1_i} = b'\cdot|\J^c| = b'\cdot|\J| \in \ZZ/2 .$$
It follows that $\ev_{(t_0,s_0)}(x)$ satisfies the condition described in Lemma~\ref{l:explicit-2} with respect to $b + b'|\J|$, and so $x \in \what{\Q}^0_{P_T}$ as desired.

To prove (2), observe that by remark~\ref{r:dim-1} the condition $P_0 \neq \{0\}$ is equivalent to $P_0 = \ZZ_w^*/(\ZZ_w^*)^2$. Lemma~\ref{l:ortho} then implies that $P^1 \subseteq \ZZ_w^*/(\ZZ_w^*)^2$. By Remark~\ref{r:dim-1} we also have $P^1 \subseteq W^w(t_1,s_1)$. Since the intersection of $W^w(t_1,s_1)$ with the subspace $\ZZ_w^*/(\ZZ_w^*)^2 \subseteq V^w$ is trivial it follows that $P^1 = \{0\}$ and hence $\loc^w(x) = 0$ for every $x \in \what{\Q}^0_{P_{T_w}}$. By claim (1) we may conclude that $\what{\Q}^0_{P_{T_w}} \subseteq \what{\Q}^0_{P_T}$. Furthermore, $\what{\Q}^0_{P_{T_w}}$ is contained in the kernel of the map 
$$ \loc^w: \what{\Q}^0_{P_T} \lrar V^w .$$ 
Since $P^0 \neq \{0\}$ this kernel is strictly smaller than $\what{\Q}^0_{P_T}$ and so we obtain a strict inclusion
$$ \what{\Q}^0_{P_{T_w}} \subsetneq \what{\Q}^0_{P_T} $$ 
as desired.
\end{proof}

\subsection{Proof of the main theorem}

In this section we will finish the proof of Theorem~\ref{t:main-uncond}. 
\begin{prop}\label{p:main-step}
Let $T$ be a finite set of places containing $S$, let $P_T = \{P_v\}_{v \in T}$ be a suitable partial adelic point and let $(t_0,s_0)$ be a $T$-admissible pair with respect to $P_T$ with associated places $\{u^0_i\}_{i \in \I}$. Let $[x] \in \what{\Q}_{P_T}$ be such that $[x] \notin \{0,[-d][p_{\I}]\}$. Then there exists a place $w$ and an extension of $P_T$ to a suitable partial adelic point $P_{T_w} = \{P_v\}_{v \in T_w}$ (with $T_w = T \cup \{w\}$) such that 
$\what{\Q}_{P_{T_w}} \subsetneq \what{\Q}_{P_T}$.
\end{prop}
\begin{proof}
Let $[x] = [c][p_{\J}] \in \what{\Q}_{P_T}$ (with $\J \subseteq \I$ and $c \in I^T$) be such that $[x] \notin \{0,[-d][p_{\I}]\}$. By our assumptions there are at least two places $v^1_\infty, v^2_\infty \in S_0$ which split in $\QQ(\sqrt{-dp_{\I}(t_0,s_0)})$. In light of Remark~\ref{r:ranks-2} we may deduce that 
$$ \dim_2\Q_{P_T} > \dim_2\what{\Q}_{P_T} \geq 2 .$$ 
We may hence conclude that there exists a $x' = [c'][p_{\J'}] \in \Q_{P_T}$ (with $\J' \subseteq \I$ and $c' \in I_T$) such that $x'$ does not belong to the group generated by $[a][p_{\cA}]$ and $[d][p_{\I}]$. We now separate our strategy into two cases:
\begin{enumerate}[(1)]
\item
If $|\J|$ is even, then by Condition (D) there exists an $i_x \in \I$ such that $c\what{D}^{\J}_{i_x} \notin \left\{1,\left[aD^{\cA}_{i_x}\right]\right\}$. In light of Remark~\ref{r:complement} we may, by possibly replacing $x$ with $x[-d][p_{\I}]$ and $x'$ with $x'[d][p_{\I}]$, assume that $i_x \in \J^c \cap (\J')^c$. By combining Corollary~\ref{c:special} with Condition (D) we may also conclude that $\left[c'D^{\J'}_{i_x}\right] \notin \left\{1,\left[aD^{\cA}_{i_x}\right]\right\}$. Applying Chebotarev's density theorem we may deduce the existence of a place $w$ such that $aD^{\cA}_{i_x}$ is a square at $w$ but $c\what{D}^{\J}_{i_x}$ and $c'D^{\J'}_{i_x}$ are non-squares at $w$. Let $t_w,s_w \in \ZZ_w$ be such that $p_{i_x}(t_w,s_w)$ is a uniformizer. 
\item
If $|\J|$ is odd then we can still use Corollary~\ref{c:special} to show that for any $i_x \in \I$ we have $\left[c'D^{\J'}_{i_x}\right] \notin \left\{1,\left[aD^{\cA}_{i_x}\right]\right\}$ (recall that $|\J'|$ is even by the definition of $\Q_{P_T}$ and the strict Selmer group). By possibly replacing $x'$ with $x'[d][p_{\I}]$ we may assume that $\J^c \cap (\J')^c \neq \emptyset$. Choose an $i_x \in \J^c \cap (\J')^c$ and apply Chebotarev's density theorem to deduce the existence of a place $w$ such that $aD^{\cA}_{i_x}$ is a square at $w$ but $c'D^{\J'}_{i_x}$ is a non-square at $w$. Let $t_w,s_w \in \ZZ_w$ be such that $p_{i_x}(t_w,s_w)$ is a uniformizer. Since $|\J|$ is odd we can, by possibly multiplying both $t_w$ and $s_w$ by a $u \in \ZZ_w^*$, assume that $[\ev_{(t_w,s_w)}(x)]$ is a non-square unit at $\ZZ_w$. 
\end{enumerate}

In either of the two cases, Lemma~\ref{l:fiber} implies that the fiber $\X_{(t_w,s_w)}$ has an $\ZZ_w$-point. It follows that we may extend $P_T$ to a suitable partial adelic point $P_{T_w} = (P_v)_{v \in T_s}$ indexed by $T_w = T \cup \{w\}$, where $P_w$ is a point lying above $(t_w,s_w)$. By Proposition~\ref{p:fibration} we may find a $T_w$-admissible pair $(t_1,s_1)$ with respect to $P_{T_w}$ and we denote by $\{u^1_i\}_{i \in \I}$ be the associated places. For $j=0,1$ consider the $\FF_2$-vector spaces $P_j$ and $P^j$ of Definition~\ref{d:P}. By our construction we have $x \in \what{\Q}^0_{P_T}$ and $x' \in \Q^0_{P_T}$. Furthermore, by Lemma~\ref{l:local-selmer} and by our choice of $w$ we have that $\loc^w(x) \neq 0 \in V^w$ and $\loc_w(x') \neq 0 \in V_w$. We may hence conclude that $P^0,P_0 \neq \{0\}$. By Proposition~\ref{p:mazur-formula} we now get 
$$ \what{\Q}^0_{P_{T_w}} \subsetneq \what{\Q}^0_{P_T} $$ 
and by Remark~\ref{r:Q_0} it follows that 
$$ \what{\Q}_{P_{T_w}} \subsetneq \what{\Q}_{P_T} $$ 
as desired.

\end{proof}

\begin{proof}[Proof of Theorem~\ref{t:main-uncond}]
Starting with a suitable partial adelic point $P_S = (P_v)_{v \in S}$ (whose existence is guaranteed by Proposition~\ref{p:suitable}) we may apply Proposition~\ref{p:fibration} to find an $S$-admissible pair $(t_0,s_0)$ with respect to $P_S$. By repeated applications of Proposition~\ref{p:main-step} we may find a finite subset of places $T$ containing $S$ and extend $P_S$ to a suitable partial adelic point $P_T$ such that for every $T$-admissible pair $(t_1,s_1)$ with respect to $P_T$ the weak dual Selmer group
$\Sel^+(\what{\T}_{(t_1,s_1)}, T_0(t_1,s_1), \{u_i\}_{i \in \I})$ is generated by $[-dp_{\I}(t_1,s_1)]$ and hence the true dual Selmer group
$$ \Sel(\what{\T}_{(t_1,s_1)}, T_0(t_1,s_1)) \subseteq \Sel^+(\what{\T}_{(t_1,s_1)}, T_0(t_1,s_1), \{u_i\}_{i \in \I} ) $$ 
is also generated by $[-dp_{\I}(t_1,s_1)]$. In light of Remark~\ref{r:control} we may apply Corollary~\ref{c:final-step} to deduce that $\X_{(t_1,s_1)}$ has an $S_0$-integral point.
\end{proof}

%\bigskip
%\footnotesize
%\noindent\textit{Acknowledgments.}
%The author would like to thank the Foundation Science Math\'emathique de Paris for their support.

\end{document}